\documentclass[12pt]{amsart}
\usepackage{amsmath}
\usepackage{amsfonts}
\usepackage{amssymb}
\usepackage{mathtools}
\usepackage{bbding}
\usepackage{stmaryrd}
\usepackage{txfonts}
\usepackage{graphicx}
\usepackage{epsfig}
\usepackage{xypic}
\usepackage{tikz}
\usepackage{longtable}
\usepackage[all]{xy}
\usepackage{pgflibraryarrows}
\usepackage{pgflibrarysnakes}
\usepackage[shortlabels]{enumitem}
\usepackage{ifpdf}
\ifpdf
\usepackage[colorlinks,final,backref=page,hyperindex]{hyperref}
\else
\usepackage[colorlinks,final,backref=page,hyperindex,hypertex]{hyperref}
\fi
\usepackage{graphicx}
\usepackage{epstopdf}
\usepackage{epsfig}
\usepackage{pdfsync}    

\usepackage{xspace}

\newtheorem{theorem}{Theorem}[section]
\newtheorem{lemma}[theorem]{Lemma}
\newtheorem{corollary}[theorem]{Corollary}

\newtheorem{conjecture}[theorem]{Conjecture}
\newtheorem{proposition}[theorem]{Proposition}

\theoremstyle{definition}
\newtheorem{definition}[theorem]{Definition}
\newtheorem{remark}[theorem]{Remark}
\newtheorem{example}[theorem]{Example}

\newcommand{\nc}{\newcommand}
\newcommand{\delete}[1]{}

\delete{

}

\def\bc{\begin{center}}
	\def\ec{\end{center}}

\nc{\tred}[1]{\textcolor{red}{#1}}
\nc{\tblue}[1]{\textcolor{blue}{#1}} \nc{\tgreen}[1]{\textcolor{green}{#1}} \nc{\tpurple}[1]{\textcolor{purple}{#1}} \nc{\btred}[1]{\textcolor{red}{\bf #1}} \nc{\btblue}[1]{\textcolor{blue}{\bf #1}} \nc{\btgreen}[1]{\textcolor{green}{\bf #1}} \nc{\btpurple}[1]{\textcolor{purple}{\bf #1}}

\newcommand{\efootnote}[1]{}

\nc{\mlabel}[1]{\label{#1}}  
\nc{\mcite}[1]{\cite{#1}}  
\nc{\mref}[1]{\ref{#1}}  
\nc{\meqref}[1]{\eqref{#1}}  
\nc{\mbibitem}[1]{\bibitem{#1}} 

\delete{
	\nc{\mlabel}[1]{\label{#1}  
		{\hfill \hspace{1cm}{\bf{{\ }\hfill(#1)}}}}
	\nc{\mcite}[1]{\cite{#1}{{\bf{{\ }(#1)}}}}  
	\nc{\mref}[1]{\ref{#1}{{\bf{{\ }(#1)}}}}  
	\nc{\meqref}[1]{\eqref{#1}{{\bf{{\ }(#1)}}}}  
	\nc{\mbibitem}[1]{\bibitem[\bf #1]{#1}} 
}

\renewcommand\geq{\geqslant}
\renewcommand\leq{\leqslant}

\renewcommand\bar[1]{\overline{#1}}

\nc{\name}[1]{{\bf #1}}

\nc{\tforall}{\quad \text{ for all }}

\nc{\mre}{\mathrm{Re}\,}

\nc{\mim}{\mathrm{im}\,}
\nc{\nz}{\varepsilon}
\nc{\Id}{\mathrm{Id}}
\nc{\Fix}{\mathrm{Fix}}

\nc{\DO}{\mathrm{DO}}
\nc{\IDO}{\mathrm{IDO}}

\nc{\mrep}{{\textit{Id}}}

\nc{\mnoindent}{\smallskip\noindent}

\nc{\bin}[2]{ (_{\stackrel{\scs{#1}}{\scs{#2}}})}  
\nc{\binc}[2]{ \left (\!\! \begin{array}{c} \scs{#1}\\
		\scs{#2} \end{array}\!\! \right )}  
\nc{\bincc}[2]{  \left ( {\scs{#1} \atop
		\vspace{-1cm}\scs{#2}} \right )}  
\nc{\bs}{\bar{S}} \nc{\cosum}{\sqsubset} \nc{\la}{\longrightarrow} \nc{\rar}{\rightarrow} \nc{\dar}{\downarrow} \nc{\dprod}{**} \nc{\dap}[1]{\downarrow \rlap{$\scriptstyle{#1}$}} \nc{\md}[1]{\bar{#1}} \nc{\uap}[1]{\uparrow \rlap{$\scriptstyle{#1}$}} \nc{\defeq}{\stackrel{\rm def}{=}} \nc{\disp}[1]{\displaystyle{#1}} \nc{\dotcup}{\ \displaystyle{\bigcup^\bullet}\ } \nc{\gzeta}{\bar{\zeta}} \nc{\hcm}{\ \hat{,}\ } \nc{\hts}{\hat{\otimes}} \nc{\barot}{{\otimes}} \nc{\free}[1]{\bar{#1}} \nc{\uni}[1]{\tilde{#1}} \nc{\hcirc}{\hat{\circ}} \nc{\leng}{\ell} \nc{\lleft}{[} \nc{\lright}{]} \nc{\lc}{\lfloor} \nc{\rc}{\rfloor}
\nc{\lb}{[} 
\nc{\rb}{]} 
\nc{\curlyl}{\left \{ \begin{array}{c} {} \\ {} \end{array}
	\right.  \!\!\!\!\!\!\!}
\nc{\curlyr}{ \!\!\!\!\!\!\!
	\left. \begin{array}{c} {} \\ {} \end{array}
	\right \} }
\nc{\longmid}{\left | \begin{array}{c} {} \\ {} \end{array}
	\right. \!\!\!\!\!\!\!}
\nc{\onetree}{\bullet} \nc{\ora}[1]{\stackrel{#1}{\rar}}
\nc{\ola}[1]{\stackrel{#1}{\la}}
\nc{\ot}{\otimes} \nc{\mot}{{{\boxtimes\,}}} \nc{\otm}{\overline{\boxtimes}} \nc{\sprod}{\bullet} \nc{\scs}[1]{\scriptstyle{#1}} \nc{\mrm}[1]{{\rm #1}} \nc{\msum}{\sum\limits}
\nc{\margin}[1]{\marginpar{\rm #1}}   
\nc{\dirlim}{\displaystyle{\lim_{\longrightarrow}}\,} \nc{\invlim}{\displaystyle{\lim_{\longleftarrow}}\,} \nc{\mvp}{\vspace{0.3cm}} \nc{\tk}{^{(k)}} \nc{\tp}{^\prime} \nc{\ttp}{^{\prime\prime}} \nc{\svp}{\vspace{2cm}} \nc{\vp}{\vspace{8cm}} \nc{\proofbegin}{\noindent{\bf Proof: }}
\nc{\proofend}{$\blacksquare$ \vspace{0.3cm}}
\nc{\modg}[1]{\!<\!\!{#1}\!\!>}
\nc{\intg}[1]{F_C(#1)} \nc{\lmodg}{\!<\!\!} \nc{\rmodg}{\!\!>\!} \nc{\cpi}{\widehat{\Pi}}
\nc{\sha}{{\mbox{\cyr X}}}  
\nc{\shap}{{\mbox{\cyrs X}}} 
\nc{\shpr}{\diamond}    
\nc{\shp}{\ast} \nc{\shplus}{\shpr^+}
\nc{\shprc}{\shpr_c}    
\nc{\msh}{\ast} \nc{\zprod}{m_0} \nc{\oprod}{m_1} \nc{\vep}{\varepsilon} \nc{\labs}{\mid\!} \nc{\rabs}{\!\mid}
\nc{\astarrow}{\overset{\raisebox{-3pt}{$\ast$}}{\rightarrow}}

\nc{\dth}{d} \nc{\mmbox}[1]{\mbox{\ #1\ }} \nc{\fp}{\mrm{FP}} \nc{\rchar}{\mrm{char}} \nc{\Fil}{\mrm{Fil}} \nc{\Mor}{Mor\xspace} \nc{\gmzvs}{gMZV\xspace} \nc{\gmzv}{gMZV\xspace} \nc{\mzv}{MZV\xspace} \nc{\mzvs}{MZVs\xspace}
\nc{\MZV}{\mathrm{MZV}}
\nc{\Hom}{\mrm{Hom}} \nc{\id}{\mrm{id}} \nc{\im}{\mrm{im}} \nc{\incl}{\mrm{incl}}  \nc{\mchar}{\rm char}

\nc{\Alg}{\mathbf{Alg}} \nc{\Bax}{\mathbf{Bax}} \nc{\bff}{\mathbf f} \nc{\bfk}{{\bf k}} \nc{\bfone}{{\bf 1}} \nc{\bfx}{\mathbf x} \nc{\bfy}{\mathbf y}
\nc{\base}[1]{\bfone^{\otimes ({#1}+1)}} 
\nc{\Cat}{\mathbf{Cat}} \delete{}
\nc{\detail}{\marginpar{\bf More detail}
	\noindent{\bf Need more detail!}
	\svp}
\nc{\Int}{\mathbf{Int}} \nc{\Mon}{\mathbf{Mon}}
\nc{\rbtm}{{shuffle }} \nc{\rbto}{{Rota-Baxter }} \nc{\remarks}{\noindent{\bf Remarks: }} \nc{\Rings}{\mathbf{Rings}} \nc{\Sets}{\mathbf{Sets}}
\nc{\balpha}{\mathbf{\alpha}}

\nc{\BA}{{\mathbb A}} \nc{\CC}{{\mathbb C}} \nc{\DD}{{\mathbb D}} \nc{\EE}{{\mathbb E}} \nc{\FF}{{\mathbb F}} \nc{\GG}{{\mathbb G}} \nc{\HH}{{\mathbb H}} \nc{\LL}{{\mathbb L}} \nc{\NN}{{\mathbb N}} \nc{\KK}{{\mathbb K}} \nc{\PP}{{\mathbb P}} \nc{\QQ}{{\mathbb Q}} \nc{\RR}{{\mathbb R}} \nc{\TT}{{\mathbb T}} \nc{\VV}{{\mathbb V}} \nc{\ZZ}{{\mathbb Z}}

\nc{\cala}{{\mathcal A}} \nc{\calc}{{\mathcal C}} \nc{\cald}{{\mathcal D}} \nc{\cale}{{\mathcal E}} \nc{\calf}{{\mathcal F}} \nc{\calg}{{\mathcal G}} \nc{\calh}{{\mathcal H}} \nc{\cali}{{\mathcal I}} \nc{\call}{{\mathcal L}} \nc{\calm}{{\mathcal M}} \nc{\caln}{{\mathcal N}} \nc{\calo}{{\mathcal O}} \nc{\calp}{{\mathcal P}} \nc{\calr}{{\mathcal R}} \nc{\cals}{{\mathcal S}} \nc{\calt}{{\mathcal T}} \nc{\calw}{{\mathcal W}} \nc{\calk}{{\mathcal K}} \nc{\calx}{{\mathcal X}}
\nc{\calz}{{\mathcal Z}}

\nc{\fraka}{{\mathfrak a}} \nc{\frakA}{{\mathfrak A}} \nc{\frakb}{{\mathfrak b}} \nc{\frakB}{{\mathfrak B}}
\nc{\frakc}{{\mathfrak c}}  \nc{\frakD}{{\mathfrak D}}
\nc{\frakH}{{\mathfrak H}}
\nc{\frakh}{{\mathfrak h}} \nc{\frakM}{{\mathfrak M}}
\nc{\frakO}{{\mathfrak O}}
\nc{\frakE}{{\mathfrak E}}
\nc{\bfrakM}{\overline{\frakM}} \nc{\frakm}{{\mathfrak m}} \nc{\frakP}{{\mathfrak P}} \nc{\frakN}{{\mathfrak N}} \nc{\frakp}{{\mathfrak p}} \nc{\frakS}{{\mathfrak S}}
\nc{\frakk}{{\mathfrak k}}
\nc{\frakx}{{\mathfrak x}}
\nc{\frakl}{{\mathfrak l}} \nc{\ox}{\bar{\frakx}} \nc{\frakX}{{\mathfrak X}} \nc{\fraky}{{\mathfrak y}} \nc\dop{\delta}
\nc{\Reduce}{{\rm Red}}

\font\cyr=wncyr10 \font\cyrs=wncyr7
\nc{\redt}[1]{\textcolor{red}{#1}}
\nc{\li}[1]{\textcolor{red}{#1}}
\nc{\lir}[1]{\textcolor{red}{Li:#1}}
\nc{\ap}[1]{\textcolor{blue}{#1}}
\nc{\apr}[1]{\textcolor{blue}{AP:#1}}

\nc{\wvec}[2]{{\scriptsize{\Big [ \!\!\begin{array}{c} #1 \\ #2 \end{array} \!\! \Big ]}}}

\nc{\bwvec}[2]{\Big(\wvec{#1}{#2}\Big)}

\nc{\jwvec}[2]{{\scriptsize{\Big [ \!\!\begin{array}{cccccccccccccc} #1 \\ #2 \end{array} \!\! \Big ]}}}
\nc{\bjwvec}[2]{\Big(\jwvec{#1}{#2}\Big)}

\begin{document}
	
\title[On differential lattices]{On differential lattices}
	
\author{Aiping Gan}
\address{School of Mathematics and Statistics,
Jiangxi Normal University, Nanchang, Jiangxi 330022, P.R. China}
\email{ganaiping78@163.com}

\author{Li Guo${}^\ast$}
\address{Department of Mathematics and Computer Science, Rutgers University, Newark, NJ 07102, USA}
\email{liguo@rutgers.edu}
	
\hyphenpenalty=8000
	
\date{\today}
	
\begin{abstract}
This paper studies the differential lattice, defined to be a lattice $L$ equipped with a map $d:L\to L$ that satisfies a lattice analog of the Leibniz rule for a derivation.
Isomorphic differential lattices are studied and classifications of differential lattices are obtained for some basic lattices.
Several families of derivations on a lattice are explicitly constructed, giving realizations of the lattice as lattices of derivations. Derivations on a finite distributive lattice are shown to have a natural structure of lattice. 
Moreover, derivations on a complete infinitely distributive lattice form a complete lattice.
For a general lattice, it is conjectured that its poset of derivations is a lattice that uniquely determines the given lattice. 
\end{abstract}
	
\subjclass[2010]{
06B10		
13N15		
06B23		
12H05		
06B15		
}
	
\keywords{lattice; derivation; congruence; differential lattice; differential algebra\\
${}^\ast$ Corresponding author: Li Guo (liguo@rutgers.edu)}
	
\maketitle
	
\tableofcontents
	
\hyphenpenalty=8000 \setcounter{section}{0}
	
	
\allowdisplaybreaks

\section {Introduction}	
\mlabel{sec:intr}

The notion of derivation from analysis has been defined for various algebraic structures by extracting the Leibniz rule
$$\frac{d}{dx}(fg)=\Big(\frac{d}{dx}(f)\Big)g +f\frac{d}{dx}(g).$$
An algebraic structure with a derivation is broadly called a differential algebra.

As the earliest instance, the differential algebra
for fields and commutative algebras has  its origin in the algebraic study of differential equations~\mcite{Ko,Ri,SP} and is a natural yet profound extension of commutative algebra and the related algebraic geometry. After many years of development, the theory has evolved into a vast area in mathematics~\mcite{CGKS,Ko,SP}. Furthermore, differential algebra has found  important applications in arithmetic geometry, logic and computational algebra, especially in the profound work of W.-T. Wu on mechanical proof of geometric theorems~\mcite{Wu,Wu2}.

Later on, there have been quite much interests in derivations for noncommutative algebras. For instance, in connection with combinatorics, differential structures were found on heap ordered trees~\mcite{GrL} and on decorated rooted trees~\mcite{GK3}.
More recently, derivations on other algebraic structures have been initiated, including for path algebras~\mcite{GL}, Lie algebras and Archimedean d-rings~\mcite{LGG,MZ,Po}. The operad of differential associative algebras was studied in~\mcite{Lod}. Furthermore, differential graded Poisson algebras have been studied~\mcite{HLW,LWZ}.
Under the nilpotent condition $d^2=0$, derivations play an essential role in homological theories~\mcite{Wei}.
	
As another major algebraic structure with broad applications, lattice theory~\mcite{Bir,bly,Gr1} has been developed in close connection with universal algebra~\mcite{BS,Gr2}. The notion of derivations on lattices was introduced by Szasz~\mcite{sz} and further developed  by Ferrari \mcite{fer}, among others. In their language, a derivation on a lattice $(L,\vee,\wedge)$ is a map $d:L\to L$ satisfying
\begin{equation}
	d(x\vee y)=d(x)\vee d(y), \quad d(x\wedge y)=(d(x)\wedge y)\vee (x\wedge d(y)) \tforall x, y\in L.
\end{equation}

More recently, the notion of derivations without the first condition was investigated by Xin and coauthors~\mcite{xin2,xin1} with motivation from information science.
They studied properties of derivations on lattices and characterized modular lattices and distributive lattices by isotone derivations. For subsequent work, see~\mcite{als1,he,jun,jan,rao,yaz}.
There are also studies on generalizations of derivations on lattices, such as generalized derivations \mcite{als2}, higher derivations \mcite{Ce1},
$f$-derivations~\mcite{CO}, $n$-derivations and $(n, m)$-derivations \mcite{Ce2}.
\smallskip 

This paper gives an algebraic study of a differential lattice, defined to be a lattice together with a derivation, applying universal algebra. Isomorphic classes of differential lattices are characterized and  classifications of differential lattices on some basic lattices are obtained. In analogy to representations of lattices as congruence lattices of algebras~\cite{GS2,KNS}, we obtain lattice structures on the set of derivations on a given lattice, suggesting that derivations on lattices can provide representations and realizations of the abstractly defined lattices. 

The paper is organized as follows. In Section~\mref{sec:der},
the notion of differential lattices is given and basic properties of lattice derivations  are reviewed and generalized (Proposition \mref{the:000}). In particular, inner derivations, isotone derivations and meet-translation derivations are shown to be equivalent.

Section~\mref{sec:iso} considers isomorphic classes of differential lattices. In Section~\mref{ss:dl}, isomorphic classes of differential lattices and isomorphic derivations on a given lattice are introduced and their basic properties are given.
We also show that there are several explicitly constructed families of derivations on any given lattice (Proposition~\mref{pro:000}).
In Section~\mref{ss:class}, we characterize derivations on two types of lattices: the finite chains and
the diamond type lattices $M_{n}$\, $(n\geq 3)$, leading to a classification of isomorphic derivations on these lattices: there are exactly $2^{n-1}$ derivations and $2^{n-1}$ isomorphic classes of derivations
 on a $n$-element chain (Theorem \mref{the:001}) and, on the diamond type lattice $M_{n}$, there are $2+\sum\limits_{k=1}^{n-2}(k+1)\binc{n-2}{k}$ derivations and
$2(n-1)$ isomorphic classes of derivations (Theorem \mref{t:00}).

Section~\mref{sec:latder} gives a detailed study of possible lattice structures on the set of derivations on a lattice. We show that the set of derivations on a finite distributive lattice has a natural lattice structure (Theorem~\mref{pro:480}) and that derivations on a complete infinitely distributive lattice form a complete lattice (Theorem~\mref{p:401}). It is also proved that isotone derivations form a lattice isomorphic to the lattice (Proposition~\mref{pp:40}). Furthermore, there is another natural family of derivations that also form a lattice isomorphic to the given lattice (Proposition \mref{pp:440}). Derivations on several families of lattices are shown to have lattice structures. 
These results provide strong evidence to the conjecture that the set of derivations on any lattice is a lattice and that a lattice is uniquely determined by its poset of derivations (Conjecture \ref{cj:00}). 

\smallskip

\noindent
{\bf Notations.}
Throughout this paper, unless otherwise specified, a lattice is assumed to be bounded  $(L, \vee, \wedge, 0, 1)$ with the bottom element $0$ and the top element $1$.
Let $|A|$  denote the cardinality of a set $A$.
For any elements $a, b$ of a poset $(A, \leq)$, we write $a< b$ if $a\leq b$ and $a\neq b$.

\section {Differential lattices and basic properties}
\mlabel{sec:der}

In this section, the notion of differential lattices is introduced and their basic properties are presented. We refer the reader to~\mcite{bly,BS,Gr1} for background on lattices.

Combining the structures of a lattice and a derivation in the language of universal algebra~\mcite{BS}, we give

\begin{definition}
	A \name{differential lattice} is an algebra $(L, \vee, \wedge, d, 0, 1)$ of type $(2, 2, 1, 0, 0)$ such that
	\begin{enumerate}
		\item $(L, \vee, \wedge, 0, 1)$ is a bounded lattice, and
		
		\item  $d$ is a \name{derivation} on $L$ in the sense that~\mcite{xin1}
\begin{equation}
	d(x\wedge y)=(d(x)\wedge y)\vee (x\wedge d(y)) \quad \tforall ~ x, y\in L.
	\mlabel{eq:der}
\end{equation}
	\end{enumerate}
\end{definition}

Adapting the classical terminology of differential algebras~\mcite{Ko}, we also call a derivation a \name{differential operator}. More generally we also call a map $f:L\to L$ an \name{operator} even though there is no linearity involved (see Remark~\mref{rk:cond}\ref{it:cond1}).

Since all axioms of differential lattices are equations between terms,
the class of all differential lattices forms a variety.
Thus the notions of isomorphism, subalgebra, congruence and direct product, etc,
are defined from the corresponding general notions in universal algebra~\mcite{BS}.

Simple examples of derivations include  the zero operator $\textbf{0}_{L}$ and the identity operator $ \mrep_{L}$ on $L$:
$$\textbf{0}_{L}: L\to L, x\mapsto 0 \quad \text{ and }\quad
 \mrep_{L}: L\to L, x\mapsto x \ \tforall x\in L.$$
Moreover, for a given $u\in L$, the map
$$d_u(x):=x\wedge u \quad \tforall x\in L,$$
is a derivation, called an \name{inner derivation}. It is called a principal derivation in~\mcite{xin1}. If $L$ is a distributive lattice, then the set of inner derivations on $L$ is a lattice that is isomorphic to $L$~\cite[Theorem~3.29]{xin1}.
The distributivity condition will be removed in Proposition~\mref{pp:40}.

We give some general remarks on our choice of conditions for a differential lattice.
\begin{remark}
\begin{enumerate}
\item Following the recent studies starting in~\mcite{xin1}, we do not impose the extra ``linearity" condition $d(x\vee y)=d(x)\vee d(y)$, in contrast to some earlier treatments~\mcite{fer,sz}. Our choice of the conditions has its motivation from information science~\mcite{als1,jun,jan,xin1} and already leads to good properties as displayed in Proposition~\mref{pro:201} and the rest of the paper. Indeed as shown in~\mcite{fer}, including the linearity would render the lattice derivation quite specialized:
$$d(x\wedge y)=d(x)\wedge y = x\wedge d(y) \quad \tforall x, y \in L$$
and consequently, a derivation with the linearity must be an inner derivation. See Proposition \mref{the:000}.
\mlabel{it:cond1}
\item
It is natural to take the notion dual to the derivation on $L$ defined by Eq.~(\mref{eq:der}) and consider the condition
$$d(x\vee y)=(d(x)\vee y) \wedge(x\vee d(y)) \quad \tforall x, y\in L.$$
If this condition is imposed alone, then the study should be completely parallel to the study of Eq.~\meqref{eq:der} due to the symmetry of the operations $\vee$ and $\wedge$ in the definition of a lattice; while if both conditions are imposed, then the study becomes trivial since $d$ has to be the identity derivation according to~\cite[Theorem 3.17]{xin1}.
\mlabel{it:cond2}
\end{enumerate}
\mlabel{rk:cond}
\end{remark}

Denote the set of all derivations on $L$ by $\DO(L)$. We recall the following result for later applications.
\begin{proposition}
	\name{\mcite{xin1}} Let $(L, \vee, \wedge,  0, 1)$ be a lattice, $d\in \DO(L)$ and
	$x, y\in L$. 
	\begin{enumerate}
		\item $d(x)\leq x$ and, in particular, $d(0)=0$.
		\mlabel{it:2011}
		\item $x\wedge d(y)\leq d(x\wedge y)$.
		\mlabel{it:2012}
		\item If $x\leq d(u)$ for some $u\in L$, then $d(x)=x$.
		\mlabel{it:2013}
		\item If $d(1)=1$, then $d=\mrep_{L}$.
		\mlabel{it:2014}
		\item $d$ is  idempotent, that is, $d^{2}=d$.
		\mlabel{it:2015}
	\end{enumerate}
\mlabel{pro:201}
\end{proposition}

Let $(L, \vee, \wedge,  0, 1)$ be a  lattice and $d$  be an operator on $L$.
Denote the set of all fix points of $d$ by
$\Fix_{d}(L)$:
$$\Fix_{d}(L):=\{x\in L~|~ d(x)=x\}\subseteq L.$$

\begin{lemma}
Let $(L, \vee, \wedge,  0, 1)$ be a lattice and $d$ be an operator on $L$. Then $d^{2}=d$ if and only if $\Fix_{d}(L)$ equals to the image $d(L)$ of $d$.
\mlabel{l:200}
\end{lemma}
\begin{proof}
For an operator $d$ on a lattice $(L, \vee, \wedge,  0, 1)$, first we have $\Fix_d(L)=d(\Fix_d(L)) \subseteq d(L)$.

If $d^{2}=d$, then  $d(L)\subseteq \Fix_{d}(L)$.
Indeed, for any $y=d(u)\in d(L)$, we have
 $d(y)=d(d(u))=d^{2}(u)=d(u)=y$, meaning $y\in \Fix_{d}(L)$. Hence  $d(L)\subseteq \Fix_{d}(L)$.
Therefore $\Fix_{d}(L)=d(L)$.

Conversely, if $\Fix_{d}(L)=d(L)$, then for any $x\in L$, we have $d(x)\in \Fix_{d}(L)$ and so
$d^{2}(x)=d(d(x))=d(x)$. Consequently, $d^{2}=d$.
\end{proof}

\begin{corollary}
Let $(L, \vee, \wedge,  0, 1)$ be a  lattice and  $d\in \DO(L)$. Then  $\Fix_{d}(L)=d(L)$.
\mlabel{c:200}
\end{corollary}
\begin{proof} It follows from Proposition \mref{pro:201} and Lemma \mref{l:200}.
\end{proof}

A derivation $d$ on a lattice $L$ is called \name{isotone} \mcite{xin1}  if $d(x)\leq d(y)$ for any $x, y\in L$ with $x\leq y$.
Denote the set of all isotone derivations on $L$ by  $\IDO(L)$.
Also recall from~\mcite{sz} that a map $d:L\to L$ is called a \name{meet-translation} if $d(x\wedge y)=x\wedge d(y)$ for all $x, y\in L$.

The following result is a simple improvement of~\cite[Theorem~3.10]{xin2} and \cite[Theorem~3.18]{xin1}, by not requiring that the operator $d$ is a derivation in the hypothesis.

\begin{proposition}
Let $(L, \vee, \wedge,  0, 1)$ be a lattice and $d$ be an operator on $L$. Then
the following statements are equivalent:
\begin{enumerate}
\item $d$ is an isotone derivation.
\mlabel{it:10001}
\item $d$ is meet-translation.
\mlabel{it:10002}
\item $d(x)= x\wedge d(1)$ for any $x\in L$.
\mlabel{it:10003}
\item $d$ is an inner derivation.
\mlabel{it:10004}
\end{enumerate}
Furthermore, these statements are implied by the linearity of a derivation:
\begin{enumerate}\addtocounter{enumi}{4}
	\item $d$ is a derivation with the {\bf linearity}
	 $d(x\vee y)=d(x)\vee d(y)$ for all $x, y\in L$.
	 \mlabel{it:10005}
\end{enumerate}
If $L$ is distributive, then all the five statements are equivalent.
\mlabel{the:000}
\end{proposition}

\begin{proof}
\mnoindent
\mref{it:10001}$\Rightarrow$\mref{it:10002}
Let $d\in \IDO(L)$ and $x, y\in L$. Since $x\wedge y\leq y$, we have $d(x\wedge y)\leq d(y)$. Also,
$d(x\wedge y)\leq x\wedge y\leq x$ by Proposition \mref{pro:201} \mref{it:2011}. It follows that  $d(x\wedge y)\leq x\wedge d(y)$ and hence
$d(x\wedge y)=x\wedge d(y)$ by Proposition \mref{pro:201} \mref{it:2012}.

 \mnoindent
\mref{it:10002}$\Rightarrow$\mref{it:10003}
Assume that \mref{it:10002} holds. Then $d(x)=d(x\wedge 1)=x\wedge d(1)$ for any $x\in L$, giving \mref{it:10003}.

\mnoindent
\mref{it:10003}$\Rightarrow$\mref{it:10004} The implication is clear.

 \mnoindent
\mref{it:10004}$\Rightarrow$\mref{it:10001} The implication follows from \cite[Example 3.8]{xin1}.

By~\mcite{fer}, a derivation $d$ with the linearity implies that $d$ is meet-translation and hence $d$ satisfies all the conditions~\mref{it:10001} -- \mref{it:10004}.

For the last statement, assume that $(L, \vee, \wedge,  0, 1)$ is a distributive lattice. Let $d\in \IDO(L)$ and $x, y\in L$. Then by the equivalence of \mref{it:10003} and \mref{it:10004}, we obtain
\[ d(x\vee y)=(x\vee y)\wedge d(1)=(x\wedge d(1))\vee (y\wedge d(1))=d(x)\vee d(y). \]
Hence condition~\mref{it:10004} implies condition~\mref{it:10005} and hence all the five conditions are equivalent.
\end{proof}

The distributivity in the last statement of the proposition cannot be removed. In particular, if a lattice $L$ is not  distributive, then $d\in \IDO(L)$ does not necessarily imply that $d(x\vee y)=d(x)\vee d(y)$ for any $x, y\in L$. For example, let $M_{5}=\{0, b_{1}, b_{2}, b_{3}, 1\}$ be the modular lattice whose Hasse diagram is
$$
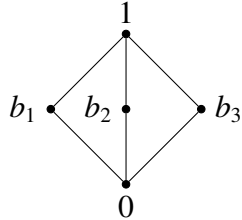
\begin{tikzpicture}
	\tikzstyle{every node}=[draw,circle,fill=black,node distance=1.0cm,
	minimum size=1.0pt, inner sep=1.0pt]
	\node[circle] (1)                        [label=above :   $1$]{};
	\node[circle] (2) [below   of=1]     [label=left : $b_{2}$]{};
	\node[circle] (3) [ left of=2]             [label=left  : $b_{1}$] {};
	\node[circle] (4) [ right   of=2]     [label=right : $b_{3}$]{};
	\node[circle] (5) [below   of=2]     [label=below: $0$] {};
	
	\draw[-] (1) --   (2); \draw[-] (1) --   (3); \draw[-] (1) --   (4);
	\draw[-] (2) --   (5); \draw[-] (3) --   (5);
	\draw[-] (4) --   (5);
\end{tikzpicture}
$$
Define an operator $d: M_{5}\rightarrow M_{5}$  by
$d(x):=x\wedge b_{1}$ for any $x\in M_{5}$. Then $d\in \IDO(M_{5})$ by Proposition \mref{the:000}, but
$d(b_{2}\vee b_{3})=d(1)=b_{1}\neq 0=d(b_{2})\vee d(b_{3})$.

\begin{corollary}
Let $(L, \vee, \wedge,  0, 1)$ be a  lattice.
\begin{enumerate}
\item There is a bijection between $\IDO(L)$ and $L$.
\mlabel{it:2001}
\item $d\in \IDO(L)$ implies that $d(x\wedge y)=d(x)\wedge d(y)$ for any $x, y\in L$.
\mlabel{it:2002}
\end{enumerate}
\mlabel{cor:200}
\end{corollary}

\begin{proof}
\mref{it:2001}
Define a map $f: \IDO(L)\rightarrow L$ by $f(d)=d(1)$
for any $d\in \IDO(L)$. Also define a map $g:L\to \IDO(L)$ by $g(u)=d_{u}$ for any $u\in L$. Then by Proposition \mref{the:000}, we have
$fg=\mrep_L$ and $gf=\mrep_{\IDO(L)}$. Hence $f$ is a bijection.

\mnoindent
\mref{it:2002} Let $d\in \IDO(L)$ and $x, y\in L$. Then
 $d(x\wedge y)=x\wedge y\wedge d(1)=(x\wedge d(1))\wedge (y\wedge d(1))=d(x)\wedge d(y)$
 by Proposition \mref{the:000}.
\end{proof}

Corollary \mref{cor:200}~\mref{it:2001} suggests that $\IDO(L)$ can be equipped with a natural lattice structure that is isomorphic to the lattice $L$. We will show that this is indeed the case in Proposition~\mref{pp:40}.

\begin{remark}
The converse of Corollary~\mref{cor:200} \mref{it:2002} does not hold.
For example, for a given $u\in L\backslash \{0\}$, define  an operator $d$ on $L$  by $d(x):=u$ for any $x\in L$.
It is clear that $d$ satisfies the condition $d(x\wedge y)=d(x)\wedge d(y)$
for any $x, y\in L$. But $d$ is not a derivation since
$d(0)=u\neq 0$.
\mlabel{re:000}
\end{remark}

\section{Isomorphic classes of differential lattices}
\mlabel{sec:iso}

In this section, we study isomorphic classes of differential lattices. In particular,
we classify isomorphic classes of differential lattices on two families of underlying lattices: the finite chains and the diamond type lattices $M_{n}$\, $(n\geq 3)$.

\subsection {Isomorphic differential lattices}
\mlabel{ss:dl}

\begin{definition}
	Two differential lattices $(L,\vee,\wedge,d, 0, 1)$ and $(L',\vee',\wedge',d',0',1')$ are called \name{isomorphic} if there is an isomorphism of lattices $f:L\to L'$
	such that $fd=d'f$.
When the lattice $L'$ is the same as $L$, we also say that $d$ is \name{isomorphic to} $d'$ and write  $d\cong d'$.

\end{definition}
Thus derivations $d$ and $d'$ on a lattice $L$ are isomorphic
if there exists a lattice automorphism $f: L\rightarrow L$ such that $fd=d'f$.

The relation $\cong$ is an equivalence relation on $\DO(L)$.
The corresponding equivalent classes are called the \name{isomorphic classes of derivations} on $L$. They are isomorphic classes of differential lattices whose underlying lattice is $L$.

\begin{remark}
There is no loss of generality in our approach of working on a fixed underlying lattice. Just observe that the classification of all isomorphic classes of differential lattices is the same as the classification of all isomorphic classes of differential lattices
on a given underlying lattice, as the underlying lattice runs through isomorphic classes of lattices.
\end{remark}

\begin{lemma} \mlabel{lll:30}
Let $(L, \vee, \wedge,  0, 1)$ be a lattice,  $d, d'\in \DO(L)$ such that $d\cong d'$. 
\begin{enumerate}
\item $d(1)=0$ if and only if  $d'(1)=0$.
\mlabel{it:31}
\item \mlabel{it:32}
$|\Fix_{d}(L)|=|\Fix_{d'}(L)|$.
 \end{enumerate}
\end{lemma}
\begin{proof}
Assume that $d, d'\in \DO(L)$
and  $d\cong d'$. Then there exists a lattice automorphism
$f: L\rightarrow L$  such that $f(d(x))=d'(f(x))$ for any $x\in L$.

\mnoindent
\mref{it:31} If $d(1)=0$, then
 $d'(1)=d'(f(1))=f(d(1))=f(0)=0$, since $f(1)=1$ and $f(0)=0$.
 By the symmetry of $d$ and $d'$ in the claim, $d'(1)=0$ implies that  $d(1)=0$. Thus \mref{it:31} holds.

\mnoindent
\mref{it:32} Since $f(d(x))=d'(f(x))$ for any $x\in L$, we have
 $f(\Fix_{d}(L))\subseteq \Fix_{d'}(L)$ by Corollary \mref{c:200} and so the restriction $f|_{\Fix_{d}(L)}$ of $f$ to $\Fix_{d}(L)$
 is an injective map from $\Fix_{d}(L)$ to $\Fix_{d'}(L)$. Thus $|\Fix_{d}(L)|\leq|\Fix_{d'}(L)|$.
By the symmetry of $d$ and $d'$, we obtain $|\Fix_{d'}(L)|\leq|\Fix_{d}(L)|$. Hence \mref{it:32} holds.
\end{proof}

We next show that the isomorphic classes of the zero derivation $\textbf{0}_L$ and the identity derivation $ \mrep_{L}$ only have one element.

\begin{lemma}
Let $(L, \vee, \wedge,  0, 1)$ be a lattice and $d\in \DO(L)$.
\begin{enumerate}
\item $d\cong  \mrep_{L}$ if and only if $d=  \mrep_{L}$.
\mlabel{it:001}
\item  $d\cong \textbf{0}_{L}$ if and only if $d= \textbf{0}_{L}$.
\mlabel{it:002}
\end{enumerate}
\mlabel{lem:00}
\end{lemma}
\begin{proof}
\mref{it:001}
 Assume that $d\cong  \mrep_{L}$. Then  there exists a lattice automorphism $f: L\rightarrow L$ such that
$df=f  \mrep_{L}=f=\mrep_{L}f $. Thus $d= \mrep_{L}$ since $f$ is bijective.

\mnoindent
\mref{it:002} Assume that $d\cong \textbf{0}_{L}$. Then  there exists a lattice automorphism $f: L\rightarrow L$ such that
 $df=f\textbf{0}_{L}$. Since $f$ is bijective and $f(0)=0$, we have $f\textbf{0}_{L}=\textbf{0}_L=\textbf{0}_{L} f$,
 and so $df=\textbf{0}_{L}f$. Thus
 $d= \textbf{0}_{L}$ since $f$ is bijective.
\end{proof}

One derivation gives rise to others as shown below.
\begin{proposition}
Let  $(L, \vee, \wedge,  0, 1)$ be a lattice and $d$ in $\DO(L)$. Give $u\in L$ with $u\leq d(1)$ and define  an operator $d'$ on $L$ by
	$$
	d'(x):=
	\begin{cases}
	u,  & \textrm{if}~ x=1; \\
	d(x),  & \textrm{otherwise}.
	\end{cases}
	$$
	Then $d'$ is in $\DO(L)$.
	\mlabel{por:111}
\end{proposition}
\begin{proof}
Let $d, d'$ and $u$ be as given in the proposition. Let $x, y\in L$.
	
	If $x, y\in L\backslash \{1\}$, then 
	$x\wedge y\in L\backslash \{1\}$ and so
	$$d'(x\wedge y)=d(x\wedge y)=(d(x)\wedge y)\vee (x\wedge d(y))=(d'(x)\wedge y)\vee (x\wedge d'(y))$$ 
	since $d\in \DO(L)$.

	If $x\in L\backslash \{1\}, y=1$, then
	since $d'(1)=u\leq d(1)$, we have
	$x\wedge d'(1)\leq x\wedge d(1)\leq d(x)$ by Proposition \mref{pro:201} \mref{it:2012} and so
	$$d'(x\wedge y)=d'(x)=d(x)=d(x)\vee (x\wedge d'(1))=(d'(x)\wedge y)\vee (x\wedge d'(y)).$$
	
	If $y\in L\backslash \{1\}, x=1$, then  we similarly have $d'(x\wedge y)=(d'(x)\wedge y)\vee (x\wedge d'(y))$.
	
	If $x=y=1$, then clearly, $d'(x\wedge y)=(d'(x)\wedge y)\vee (x\wedge d'(y))$.
	
	Thus we conclude that  $d'$ is in $\DO(L)$.
\end{proof}

We next see that there is a good supply of explicitly defined derivations on any lattice.
\begin{proposition}\mlabel{pro:000}
Let $(L, \vee, \wedge,  0, 1)$ be a lattice and $u\in L$.
\begin{enumerate}
\item Define  operators $\chi^{(u)}$ and
 $\eta^{(u)}$ on $L$ as follows:
$$\chi^{(u)}(x):=
    \begin{cases}
      u,  & \textrm{if}~ x=1; \\
      x,  & \textrm{otherwise}.
    \end{cases}\quad and \quad
    \eta^{(u)}(x):=
    \begin{cases}
      u,  & \textrm{if}~ u\leq x; \\
      x,  & \textrm{otherwise}.
    \end{cases}$$
Then $\chi^{(u)}$ and $\eta^{(u)}$ are in $\DO(L)$.
In particular, $\eta^{(0)}=\textbf{0}_{L}$
and $\chi^{(1)}=\eta^{(1)}=\mrep_{L}$.
\mlabel{it:0001}
\item \mlabel{it:0002}
Define an operator $\lambda^{(u)}$
 on $L$ by:
$$ \lambda^{(u)}(x)=
    \begin{cases}
      x,  & \textrm{if}~ x\leq u; \\
      0,  & \textrm{otherwise}.
    \end{cases}.
    $$
Then $\lambda^{(u)}$ is in $\DO(L)$ if and only if
 $L$ satisfies the condition:
\begin{equation}
	\mlabel{eq:con}
	\tforall ~ x, y\in L,	x\nleq u ~~ and ~~ y\nleq u ~~imply ~~ that ~~ x\wedge y\nleq u ~~or ~~x\wedge y=0.
\end{equation}
\end{enumerate}
\end{proposition}
\begin{proof}
\mref{it:0001}
Since $\mrep_{L}$ is in $\DO(L)$ and $u\leq 1=\mrep_{L}(1)$, we have
 $\chi^{(u)}\in \DO(L)$ by Proposition \mref{por:111}.

To prove that $\eta^{(u)}$ is in $\DO(L)$, let $x, y\in L$ and distinguish several cases.

If  $u \nleq x$, then $\eta^{(u)}(x)=x$ and $u\nleq x\wedge y$.  It follows that
$\eta^{(u)}(x\wedge y)=x\wedge y$. Noting that
 $ \eta^{(u)}(y)\leq y$, we obtain
$$\eta^{(u)}(x\wedge y)=x\wedge y=(x\wedge y)\vee (x\wedge \eta^{(u)}(y))=(\eta^{(u)}(x)\wedge y)\vee (x\wedge \eta^{(u)}(y)).$$

If  $u\nleq y$, then we similarly get 
$$\eta^{(u)}(x\wedge y)=(\eta^{(u)}(x)\wedge y)\vee (x\wedge \eta^{(u)}(y)).$$

If $u\leq x$ and $u\leq y$,  then $\eta^{(u)}(x)=\eta^{(u)}(y)=u$ and
 $u\leq x\wedge y$, which implies 
$$\eta^{(u)}(x\wedge y)=u=(u\wedge y)\vee (x\wedge u)=(\eta^{(u)}(x)\wedge y)\vee (x\wedge \eta^{(u)}(y)).$$

Therefore, we conclude that $\eta^{(u)}$ is in $\DO(L)$.

\mnoindent
\mref{it:0002}  Assume that $L$ satisfies condition (\mref{eq:con}).
By definition, $ \lambda^{(u)}(w)\leq w$ for any $w\in L$.
 To show that  $\lambda^{(u)}\in \DO(L)$,
let $x, y\in L$ and distinguish several cases.

If $x\leq u$, then $\lambda^{( u)}(x)=x$ and $x\wedge y\leq u$. Since
$ \lambda^{(u)}(y)\leq y$, we have 
$$\lambda^{(u)}(x\wedge y)=x\wedge y=(x\wedge y)\vee (x\wedge \lambda^{( u)}(y))=
(\lambda^{(u)}(x)\wedge y)\vee (x\wedge \lambda^{( u)}(y)).$$

If $y\leq u$, then we similarly obtain
$$\lambda^{(u)}(x\wedge y)=(\lambda^{( u)}(x)\wedge y)\vee (x\wedge \lambda^{(u)}(y)).$$

If $x\nleq u$ and $y\nleq u$,  then
$\lambda^{( u)}(x)=\lambda^{( u)}(y)=0$, and either
$x\wedge y\nleq  u$ or $x\wedge y=0$ by condition (\mref{eq:con}). It follows that
$$\lambda^{(u)}(x\wedge y)=0=(0\wedge y)\vee (x\wedge 0)=(\lambda^{( u)}(x)\wedge y)\vee (x\wedge \lambda^{( u)}(y)).$$
Therefore, we get $\lambda^{(u)}\in \DO(L)$.

Conversely, suppose that $\lambda^{(u)}$ is in $\DO(L)$. To show that $L$ satisfies
condition (\mref{eq:con}), let $x, y\in L$ such that $x\nleq u$ and $ y\nleq u$.
Then $\lambda^{( u)}(x)=\lambda^{( u)}(y)=0$ and so
$$\lambda^{(u)}(x\wedge y)=(\lambda^{( u)}(x)\wedge y)\vee (x\wedge \lambda^{( u)}(y))=0.$$ This implies
$x\wedge y\nleq  u$ or $x\wedge y=0$ by the definition of $\lambda^{(u)}$. Therefore $L$ satisfies condition (\mref{eq:con}).
\end{proof}

\begin{example}
\begin{enumerate}
\item
 Let $L$ be a chain and $u\in L$. It is clear that $L$ satisfies
condition (\mref{eq:con}). So
$\lambda^{(u)}\in \DO(L)$ by Proposition \mref{pro:000}.

\item Let $B_{8}=\{0, a, b, c, u, v, w, 1\}$ be the $8$-element Boolean lattice whose Hasse diagram is given by
$$
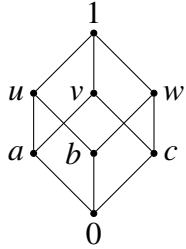
\begin{tikzpicture}
\tikzstyle{every node}=[draw,circle,fill=black,node distance=0.8cm,
minimum size=0.8pt, inner sep=0.8pt]
\node[circle] (1)                        [label=above :   $1$]{};
\node[circle] (2)   [below   of=1]                      [label=left :   $v$]{};
\node[circle] (3) [ left of=2]     [label=left : $u$]{};
\node[circle] (4) [ right of=2]             [label=right  : $w$] {};
\node[circle] (5)   [below   of=3]                      [label=left :   $a$]{};
\node[circle] (6) [below   of=2]     [label=left : $b$]{};
\node[circle] (7) [below  of=4]             [label=right  : $c$] {};
\node[circle] (8) [below   of=6]     [label=below: $0$] {};

\draw[-] (1) --   (2); \draw[-] (1) --   (3); \draw[-] (1) --   (4);
 \draw[-] (3) --   (5); \draw[-] (3) --   (6);
\draw[-] (2) --   (5); \draw[-] (2) --   (7); \draw[-] (4) --   (6);
\draw[-] (4) --   (7);  \draw[-] (5) --   (8); \draw[-] (6) --   (8);
\draw[-] (7) --   (8);
\end{tikzpicture}
$$
It is easy to verify that $B_{8}$ satisfies condition (\mref{eq:con}) for $u, v, w$,
while $B_{8}$ does not satisfy condition (\mref{eq:con}) for $a, b, c$.
So by Proposition \mref{pro:000}, $\lambda^{(u)}, \lambda^{(v)}$ and $\lambda^{(w)}$ are in $\DO(B_{8})$,
 but  $\lambda^{(a)}, \lambda^{(b)}$ and $\lambda^{(c)}$ are not in $\DO(B_{8})$.
 \end{enumerate}
 \mlabel{exm:000}
\end{example}

\begin{lemma}
Let $(L, \vee, \wedge,  0, 1)$ be a lattice. 
\begin{enumerate}
\item  $\chi^{(0)}\not\cong d$ for any $d\in \DO(L)$ with $d(1)\neq 0$. In particular,
$\chi^{(0)}\not\cong \eta^{(u)}$, $\chi^{(0)}\not\cong \chi^{(u)}$ and $\chi^{(0)}\not\cong d_{u}$ for any $u\in L\backslash \{0\}$.
\mlabel{it:3001}
\item If $|L|\geq 4$, then $\chi^{(u)}\not\cong d_{v}$ or $\chi^{(v)}\not\cong d_{u}$
for any  $u, v\in L\backslash\{0, 1\}$ with $u\neq v$.
\mlabel{it:3002}
\end{enumerate}
\mlabel{lem:3003}
 \end{lemma}
 \begin{proof}
\mref{it:3001}
  Since $\chi^{(0)}(1)=0$, Lemma~\mref{lll:30} gives $\chi^{(0)}\not\cong d$ for any $d\in \DO(L)$ with $d(1)\neq 0$, which implies that
 $\chi^{(0)}\not\cong \eta^{(u)}$, $\chi^{(0)}\not\cong \chi^{(u)}$ and $\chi^{(0)}\not\cong d_{u}$ for any $u\in L\backslash \{0\}$,
 since  $\eta^{(u)}(1)=\chi^{(u)}(1)=d_{u}(1)=u\neq 0$.

\mnoindent
\mref{it:3002} Assume that $|L|\geq 4$ and let $u, v\in L\backslash\{0, 1\}$ with $u\neq v$.

Suppose that $\chi^{(u)}\cong d_v$ and $\chi^{(v)}\cong d_u$. Then $\chi^{(u)}\cong d_{v}$ means that there exists a lattice automorphism
$f: L\rightarrow L$ such that $f(\chi^{(u)}(x))=d_{v}(f(x))$ for any $x\in L$.
Since $f(0)=0$ and $f(1)=1$, there
exists $a\in L\backslash \{0, 1\}$ such that $f(a)=u$. Consequently, we have $\chi^{(u)}(a)=a$ and so
$u=f(a)=f(\chi^{(u)}(a))=d_{v}(f(a))=d_{v}(u)=u\wedge v$. Thus $u\leq v$.

Similarly,  $\chi^{(v)}\cong d_{u}$ implies $v\leq u$. Then $u=v$, a contradiction. Therefore, statement~\mref{it:3002} holds.
\end{proof}

\begin{corollary}
Let $(L, \vee, \wedge,  0, 1)$ be a lattice. 
\begin{enumerate}
\item  If $|L|\geq 3$, then there are at least four
isomorphic classes of derivations on $L$.
\mlabel{it:20001}
\item If $|L|\geq 4$, then there are at least five
isomorphic classes of derivations on $L$.
\mlabel{it:20002}
\end{enumerate}
\mlabel{cor:2000}
\end{corollary}
\begin{proof}

\mnoindent
\mref{it:20001}
 Assume that $|L|\geq 3$ and let $u\in L\backslash \{0, 1\}$.
Then $\chi^{(0)}$, $\eta^{(u)}$ are in $\DO(L)$ by Proposition \mref{pro:000}
and $\chi^{(0)}\not\cong \eta^{(u)}$ by Lemma \mref{lem:3003}. Moreover, it is easy to see that
$\chi^{(0)}\neq \textbf{0}_{L}$,  $\chi^{(0)}\neq \mrep_{L}$, $\eta^{(u)}\neq \textbf{0}_{L}$  and $\eta^{(u)}\neq \mrep_{L}$. Consequently,
we have by Lemma \mref{lem:00} that
 $\mrep_{L}, \textbf{0}_{L}, \chi^{(0)}$ and $ \eta^{(u)}$
 are mutually non-isomorphic derivations  on $L$.

\mnoindent
\mref{it:20002}
Assume that   $|L|\geq 4$ and let $u, v\in L\backslash \{0, 1\}$ with $u\neq v$. By Lemma \mref{lem:3003}, we have
$\chi^{(u)}\not\cong d_{v}$ or $\chi^{(v)}\not\cong d_{u}$. Without loss of generality, suppose $\chi^{(u)}\not\cong d_{v}$. It
follows by Proposition \mref{pro:000}, Lemma \mref{lem:00} and Lemma \mref{lem:3003} that
  $\mrep_{L}, \textbf{0}_{L}, \chi^{(0)}, \chi^{(u)}$ and $d_{v}$
are mutually non-isomorphic derivations on $L$.
\end{proof}

\subsection{Classification of differential lattices}
\mlabel{ss:class}

We next apply the general results above to classify all derivations on finite chains and diamond type lattices $M_{n}$.
\subsubsection{Classification on finite chains}
\begin{lemma}
Let $(L, \vee, \wedge,  0, 1)$ be a finite chain and $d, d'\in \DO(L)$. Then
 $d\cong d'$ if and only if $d= d'$.
 \mlabel{lem:11}
\end{lemma}
\begin{proof}
Assume that $(L, \vee, \wedge,  0, 1)$ is a finite chain and $d, d'\in \DO(L)$.
Certainly $d= d'$ implies $d\cong d'$.

Conversely, suppose that   $d\cong d'$. Then  there exists a lattice automorphism $f: L\rightarrow L$ such that
$df=fd'$.  Since $f$ is a bijection and both $f$ and $f^{-1}$ are order-preserving
(see Theorem 2.3 in \mcite{BS}), we have $f=\mrep_{L}$ and so   $d=df=fd'=d'$.
\end{proof}

\begin{remark}
On the other hand, if $(L, \vee, \wedge,  0, 1)$ is an infinite chain, then for $d, d'\in \DO(L)$, 
 $d\cong d'$ does not necessarily imply  $d= d'$.

For example,  equip the unit interval $[0, 1]$ with the usual order $\leq$.
 Then $([0, 1], \leq)$ is a chain.
Consider  inner derivations $d_\frac{1}{2}$ and $d_\frac{1}{4}$,
we have $d_\frac{1}{2}\neq d_\frac{1}{4}$ since $d_\frac{1}{2}(1)=\frac{1}{2}\neq \frac{1}{4}=d_\frac{1}{4}(1)$. However, $d_\frac{1}{2}\cong d_\frac{1}{4}$.
 In fact, let $f: [0, 1]\rightarrow [0, 1]$ be defined by $f(x)=x^{2}$ for any $x\in [0, 1]$. Then
 it is easy to see that $f$ is a bijection and both $f$ and $f^{-1}$ are order-preserving, so $f$ is a
 lattice isomorphism by \cite[Theorem 2.3]{BS}. Also, we have
 $f(d_\frac{1}{2}(x))=f(x\wedge \frac{1}{2})=(x\wedge \frac{1}{2})^{2}=x^{2}\wedge \frac{1}{4}=d_\frac{1}{4}(f(x))$
 for any $x\in [0, 1]$. Thus $fd_\frac{1}{2}=d_\frac{1}{4}f$ and hence
  $d_\frac{1}{2}\cong d_\frac{1}{4}$.
\end{remark}

\begin{proposition}
Let $(L, \vee, \wedge,  0, 1)$ be a lattice. 
\begin{enumerate}
\item  $\DO(L)=\{\mrep_{L}, \textbf{0}_{L}\}$ if and only if $L$ is a $2$-element chain.
\mlabel{it:01}
\item  $|\DO(L)|=4$ if and only if $L$ is a $3$-element chain.
\mlabel{it:02}
\end{enumerate}
\mlabel{p:2000}
\end{proposition}
\begin{proof}
\mref{it:01}
 Assume that $L=\{0, 1\}$ and $d\in \DO(L)$. Then $d(0)=0$ by Proposition \mref{pro:201}
\mref{it:2011}, which implies that $d=\mrep_{L}$ if
 $d(1)=1$, and $d=\textbf{0}_{L}$ if $d(1)=0$. Therefore $\DO(L)=\{\mrep_{L}, \textbf{0}_{L}\}$.

Conversely, assume that  $|L|\geq 3$.
Then  $\DO(L)\neq\{\mrep_{L}, \textbf{0}_{L}\}$ by Corollary \mref{cor:2000} \mref{it:20001}.

\mnoindent
\mref{it:02}
Assume that $L=\{0, u, 1\}$ is a $3$-element chain with $0< u<1$, and $d\in \DO(L)$. Then $d(0)=0$ by Proposition \mref{pro:201}
\mref{it:2011}.

If $d(1)=1$, then $d=\mrep_{L}$ by Proposition \mref{pro:201} \mref{it:2014}.
If $d(1)=u$, then $d(u)=u$ by Proposition \mref{pro:201}
\mref{it:2013} and so $d=d_{u}$.
If $d(1)=0$ and $d(u)=u$, then $d=\chi^{(0)}$.
If $d(1)=0$ and $d(u)=0$, then $d=\textbf{0}_{L}$.
Therefore $\DO(L)=\{\mrep_{L}, \textbf{0}_{L}, \chi^{(0)}, d_{u}\}$ and $|\DO(L)|=4$.

Conversely, assume that $|\DO(L)|=4$. If $|L|\geq 4$, then $|\DO(L)|\geq 5$ by Corollary \mref{cor:2000} \mref{it:20002}, a contradiction.
Thus $|L|\leq 3$. But $|L|= 2$ implies that $|\DO(L)|=2$ by \mref{it:01}. Therefore $|L|=3$ and, consequently, $L$ is a $3$-element chain.
\end{proof}

\begin{lemma}
Let $(L, \vee, \wedge,  0, 1)$ be a chain and $u, v\in L$ with $v\leq u$. Define an operator  $\lambda^{(v; u)}$ on $L$ by
$$
    \lambda^{(v; u)}(x)=
    \begin{cases}
      x,  & \textrm{if}~ x\leq u; \\
      v,  & \textrm{otherwise}.
    \end{cases}
    $$
Then  $\lambda^{(v; u)}$ is in $\DO(L)$.
\mlabel{lem:0000}
\end{lemma}

As special cases,  $\lambda^{(v; 1)}=\mrep_{L}$ and $\lambda^{(0; u)}=\lambda^{(u)}$, where $\lambda^{(u)}$ is defined in Proposition \ref{pro:000}.

\begin{proof}
Assume that $(L, \vee, \wedge,  0, 1)$ is a chain and $u, v\in L$ with $v\leq u$. Since $L$ is a chain, we have
 $ \lambda^{(v; u)}(w)\leq w$ for any $w\in L$. To prove that  $\lambda^{(v; u)}$ is in $\DO(L)$,
let $x, y\in L$.

If $x\leq u$, then $\lambda^{(v; u)}(x)=x$ and $x\wedge y\leq u$. Then $\lambda^{(v; u)}(y)\leq y$ gives
$$\lambda^{(v; u)}(x\wedge y)=x\wedge y=(x\wedge y)\vee (x\wedge \lambda^{(v; u)}(y))=
(\lambda^{(v; u)}(x)\wedge y)\vee (x\wedge \lambda^{(v; u)}(y)).$$

The case for $y\leq u$ is similarly verified.

If $x> u$ and $y> u$,  then $\lambda^{(v; u)}(x)=\lambda^{(v; u)}(y)=v$ and
$x\wedge y> u$, which implies that
$$\lambda^{(v; u)}(x\wedge y)=v=(v\wedge y)\vee (x\wedge v)=(\lambda^{(v; u)}(x)\wedge y)\vee (x\wedge \lambda^{(v; u)}(y)).$$
Therefore $\lambda^{(v; u)}$ is in $\DO(L)$.
\end{proof}

\begin{remark}
If $L$ is not a chain, then $\lambda^{(v; u)}$ does not necessarily belong to $\DO(L)$. For example, let $N_{5}=\{0, v, u, w, 1\}$ be the pentagon lattice with the Hasses diagram:
\vspace{-.3cm}
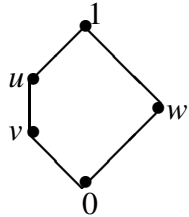
\begin{center}
\setlength{\unitlength}{0.7cm}
\begin{picture}(4,5)
\thicklines
\put(2.0,1.0){\line(-1,1){1.0}}
\put(2.0,1.0){\line(1,1){1.4}}
\put(1.0,2.0){\line(0,1){1.0}}
\put(2.0,4.0){\line(-1,-1){1.0}}
\put(2.1,4.0){\line(1,-1){1.4}}
\put(2.0,0.5){$0$}
\put(0.6,1.9){$v$}
\put(3.6,2.3){$w$}
\put(0.6,2.9){$u$}
\put(2.1,4.1){$1$}
\put(1.92,0.95){$\bullet$}
\put(0.92,1.9){$\bullet$}
\put(0.92,2.9){$\bullet$}
\put(3.3,2.35){$\bullet$}
\put(1.92,3.92){$\bullet$}
\end{picture}
\end{center}
\vspace{-.3cm}
 Then $\lambda^{(v; u)}$ is not in $\DO(N_{5})$ since
$\lambda^{(v; u)}(w)=v\nleq w$.
\end{remark}

\begin{lemma}
Let  $d$ be an operator on a  chain $L$. Suppose that $ \max_{x\in L}\{d(x)\}$ exists and denote it by $u$. Then
  $d\in \DO(L)$ if and only if $d$ satisfies the following conditions:
\begin{enumerate}
\item  $d(x)=x$ for each $x\leq u$, and
\mlabel{it:00011}
\item $d(v)\leq d(w)$ for any $v, w\in L$ with $u< w\leq v$.
\mlabel{it:00012}
\end{enumerate}
\mlabel{lem:0001}
\end{lemma}
\begin{proof}
Let $L, d$ and $u$ be as given.

If $d$ is in $\DO(L)$, then $d(x)=x$ for each $x\leq u$ by Proposition \mref{pro:201}. Thus \mref{it:00011} holds.
To prove that \mref{it:00012} holds, let $v, w\in L$ with $u< w\leq v$. Then $d(v)\leq \max_{x\in L}\{d(x)\}=u< w$ and so
$d(v)=d(v)\wedge w\leq d(v\wedge w)=d(w)$ by Proposition \mref{pro:201} \mref{it:2012}. This proves \mref{it:00012}.

Conversely, suppose that $d$ satisfies conditions \mref{it:00011} and \mref{it:00012}.
Then it is easy to see that $d(x)\leq x$
for any $x\in L$. Let $v, w\in L$ and distinguish the following cases.

If $v\leq u$, then $v\wedge w\leq u$ and so $d(v)=v$ and $d(v\wedge w)=v\wedge w$ by condition \mref{it:00011}. It follows that
 $$d(v\wedge w)=v\wedge w=(v\wedge w)\vee (v\wedge d(w))=(d(v)\wedge w)\vee (v\wedge d(w))$$
since $d(w)\leq w$.
The case for $w\leq u$ is similarly verified.

If $u<  w\leq v$, then $d(v)\leq d(w)$ by condition \mref{it:00012}. Also, since $d(w)\leq \max_{x\in G}\{d(x)\}=u< v$ and
$d(v)\leq \max_{x\in G}\{d(x)\}=u< w$,
we have
$$d(v\wedge w)=d(w)=d(v)\vee d(w)=(d(v)\wedge w)\vee (v\wedge d(w)).$$

If $u<  v\leq w$, then we similarly derive $d(v\wedge w)=(d(v)\wedge w)\vee (v\wedge d(w))$.

Thus we conclude that $d$ is in $\DO(L)$.
\end{proof}

\begin{proposition}
Let  $d$ be a derivation on a  chain $(L, \vee, \wedge,  0, 1)$. Then $d$ is in $\IDO(L)$ if and only if
 $ \max_{x\in L}\{d(x)\}$ exists and is $d(1)$.
   \mlabel{pr:111}
\end{proposition}
\begin{proof}
Assume that $(L, \vee, \wedge,  0, 1)$ is a  chain and $d\in \DO(L)$.

If $d\in \IDO(L)$, then $d(x)\leq d(1)$ for any $x\in L$ and so $d(1)= \max_{x\in L}\{d(x)\}$.

Conversely, if
 $ \max_{x\in L}\{d(x)\}$ exists and  $d(1)= \max_{x\in L}\{d(x)\}$, then by Lemma \mref{lem:0001}, we have
 $d(x)=x$ for any $x\leq d(1)$ and $d(1)\leq d(w)\leq d(1)$ for any $d(1)< w\leq 1$.
 It follows that  $d(x)=x\wedge d(1)$ for any $x\in L$ and hence $d$ is in $\IDO(L)$  by Proposition \mref{the:000}.
\end{proof}

We recall the following cominatorial lemma before the first classification theorem.
Its proof can be found for example in  \textit{https://math.stackexchange.com/questions/2231965/count-number-of-increasing-functions-nondecreasing-functions-f-1-2-3-ld.}

\begin{lemma}
	For any positive integers $k$ and $\ell$, the number of isotones from $[k]$ to $[\ell]$ is $\bincc{k+\ell-1}{k}$. The same is true for the number of antitones.
\mlabel{lem:antinum}
\end{lemma}

\begin{theorem}
Let $(L, \vee, \wedge,  0, 1)$ be an $n$-element chain. Then $|\DO(L)|=2^{n-1}$ and
 there are exactly $2^{n-1}$ isomorphic classes of derivations  on $L$.
 \mlabel{the:001}
\end{theorem}

\begin{proof}
It is clear that $|\DO(L)|=1=2^{0} $ if $n=1$.
Also, Proposition \mref{p:2000} tells us that
 $|\DO(L)|=2^{n-1} $ if $n=2$ or $3$. So we assume that $n\geq 4$.

 Let
  $L=\{0, a_{1}, a_{2}, \cdots, a_{n-2}, 1\}$  with $0< a_{1}< a_{2}< \cdots< a_{n-2}< 1$. Denote $a_{n-1}:=1$ and $a_0:=0$. By Proposition \mref{pro:201}, any
$d\in \DO(L)$ can be obtained as follows. There is $0\leq i\leq n-1$, such that $d(a_j)=a_j$ for $0\leq j\leq i$ and $d$ is an antitone from $\{a_{i+1},\cdots,a_{n-1}\}$ (of cardinality $n-i-1$) to $\{a_0,\ldots,a_{i}\}$ (of cardinality $i+1$). By Lemma~\mref{lem:antinum}, the number of such antitones is $\bincc{n-i-1+i+1-1}{n-i-1}=\bincc{n-1}{i}$. Thus the cardinality of $\DO(L)$ is
$ \sum\limits_{i=0}^{n-1} \bincc{n-1}{i} = 2^{n-1}$ which is also the number of isomorphic classes of derivations on $L$ by Lemma \mref{lem:11}.
\end{proof}

\subsubsection{Classification on diamond type lattices}

Let $n\geq 3$ be a positive integer and let $M_{n}=\{0, b_{1}, b_{2}, \cdots, b_{n-2}, 1\}$ be the diamond type lattice with Hasse diagram
\vspace{-.5cm}
\begin{center}
\setlength{\unitlength}{0.7cm}
\begin{picture}(4,5)
\thicklines
\put(2.0,1.0){\line(-2,1){2.5}}
\put(2.0,1.0){\line(-1,1){1.2}}
\put(2.0,1.0){\line(0,1){1.2}}
\put(2.0,1.0){\line(1,1){1.2}}
\put(2.0,1.0){\line(2,1){2.5}}
\put(2.0,1.0){\line(4,1){4.6}}
\put(2.0,3.5){\line(-2,-1){2.5}}
\put(2.0,3.5){\line(-1,-1){1.2}}
\put(2.0,3.5){\line(0,-1){1.2}}
\put(2.0,3.5){\line(1,-1){1.2}}
\put(2.0,3.5){\line(2,-1){2.5}}
\put(2.0,3.5){\line(4,-1){4.6}}
\put(2.0,0.5){$0$}
\put(-1.2,2.1){$b_{1}$}
\put(0.1,2.1){$b_{2}$}
\put(1.3,2.1){$b_{3}$}
\put(2.5,2.1){$b_{4}$}
\put(3.4,2.1){$\cdots$}
\put(5.0,2.1){$\cdots$}
\put(6.9,2.1){$b_{n-2}$}
\put(2.0,3.7){$1$}
\put(1.9,0.9){$\bullet$}
\put(-0.6,2.1){$\bullet$}
\put(0.6,2.1){$\bullet$}
\put(1.9,2.1){$\bullet$}
\put(3.1,2.1){$\bullet$}
\put(4.3,2.1){$\bullet$}
\put(6.5,2.1){$\bullet$}
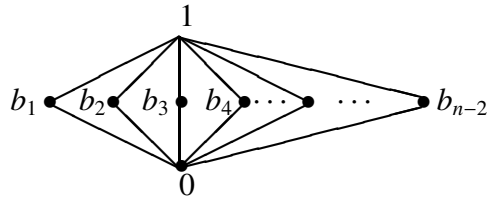
\put(0.5,0.0){Diamond type lattice $M_{n}$}
\end{picture}
\end{center}
In the rest of this section, we will determine isomorphic classes of derivations on $M_{n}$. We first give a simple characterization of derivations on $M_n$.

\begin{lemma}
	Let  $d$ be an operator on the lattice $M_{n}$ such that $1\not\in  \Fix_{d}(M_{n})$. Then
	$d\in \DO(M_{n})$ if and only if
	\begin{enumerate}
		\item  $d(1)\in \Fix_{d}(M_{n})$ and
		\mlabel{it:01001}
		\item $d(w)=0$ for each $w\in M_n\backslash (\Fix_d(M_n)\cup\{1\})$.
		\mlabel{it:01002}
	\end{enumerate}
	\mlabel{lem:0100}
\end{lemma}

\begin{proof}
If $d\in \DO(M_{n})$, then $d(1)\in d(L)=\Fix_{d}(M_{n})$ by Corollary \mref{c:200}, giving \mref{it:01001}. Also, for any
$w\in M_n\backslash (\Fix_d(M_n)\cup\{1\})$,
we have  $d(w)\leq w$ by Proposition \mref{pro:201} and so $d(w)=0$, since
$w\not\in \Fix_{d}(M_{n})\cup \{1\}$ and $d(w)\in d(L)=\Fix_{d}(M_{n})$. Thus \mref{it:01002} holds.

Conversely, suppose that $d$ satisfies conditions \mref{it:01001} and \mref{it:01002}, and  $1\not\in  \Fix_{d}(M_{n})$.
It is easy to see that $d(0)=0$ and $d(x)\leq x$
for any $x\in M_{n}$ and so $d(x\wedge x)=d(x)=(d(x)\wedge x)\vee (x\wedge d(x))$.
So to verify that $d$ is in $\DO(M_{n})$, we only need to verify the Leibniz rule in \eqref{eq:der} for $x, y\in M_{n}$ with $x\neq y$.

If $x\wedge y=0$, then $d(x)\wedge y=x\wedge d(y)=0$ since $d(x)\leq x$ and $d(y)\leq y$, which implies that
$d(x\wedge y)=d(0)=0=(d(x)\wedge y)\vee (x\wedge d(y))$.

If $x\wedge y\neq 0$, then either $x=1$ and $y\in \{ b_{1}, b_{2}, \cdots, b_{n-2}\}$, or $y=1$ and $x\in \{ b_{1}, b_{2}, \cdots, b_{n-2}\}$. Without loss of generality, assume that $x=b_{1}$ and $y= 1$. Then we have
 $d(x\wedge y)=d(x)$ and $(d(x)\wedge y)\vee (x\wedge d(y))=d(x)\vee (x\wedge d(1))$. There are the following cases to consider:
\begin{enumerate}
 \item[$\bullet$]  $d(1)=x$. Then $d(x)=x$ by condition \mref{it:01001}, which implies that
  $d(x\wedge y)=d(x)=x=d(x)\vee (x\wedge d(1))=(d(x)\wedge y)\vee (x\wedge d(y))$.
  \item[$\bullet$] $d(1)\neq x$. Since $d(1)\in \Fix_{d}(M_{n})$ but $1\not\in \Fix_{d}(M_{n})$, we have $d(1)\neq 1$ and so
   $d(1)\wedge x=0$, which implies that
  $d(x\wedge y)=d(x)=d(x)\vee (x\wedge d(1))=(d(x)\wedge y)\vee (x\wedge d(y))$.
\end{enumerate}

Therefore, we conclude that $d\in \DO(M_{n})$.
\end{proof}

We now characterize when two derivations on $M_n$ are isomorphic.

\begin{lemma}
	Let $d, d'\in \DO(M_{n})$. Then $d\cong d'$ if and only if
	\begin{enumerate}
		\item  $|\Fix_d(M_n)|=|\Fix_{d'}(M_n)|$ and
		\mlabel{it:01011}
		\item $d(1)$ and $d'(1)$ are either both zero or both nonzero.
		\mlabel{it:01012}
	\end{enumerate}
\mlabel{lem:0101}
\end{lemma}

\begin{proof}
Let  $d, d'\in \DO(M_{n})$.
If $d\cong d'$, then by Lemma \mref{lll:30},  conditions \mref{it:01011} and \mref{it:01012} hold.

Conversely, assume that $d, d'$ satisfy conditions \mref{it:01001} and \mref{it:01002}.

If $1\in \Fix_d(M_n)$, that is, $d(1)=1$, then $d=\mrep_{M_{n}}$ by Proposition \mref{pro:201} and so
$\Fix_d(M_n)=M_{n}$. It follows from condition \mref{it:01001} that $\Fix_{d'}(M_n)=M_{n}$ and thus
$d'=\mrep_{M_{n}}$. Therefore $d\cong d'$. Similarly, if  $1\in \Fix_{d'}(M_n)$, then we obtain $d=d'=\mrep_{M_{n}}$.

If $1\not\in \Fix_d(M_n)$ and $1\not\in \Fix_{d'}(M_n)$, then since $d(0)=d'(0)=0 $, we can assume
by  condition \mref{it:01001} that
\begin{center}
 $\Fix_{d}(M_{n})=\{0,  b_{i_{1}}, b_{i_{2}}, \cdots, b_{i_{k}}\}$ and
$\Fix_{d'}(M_{n})=\{0,  b_{j_{1}}, b_{j_{2}}, \cdots, b_{j_{k}}\}$,
\end{center}
where  $1\leq i_{1}<i_{2}<\cdots< i_{k}\leq n-2$ and $1\leq j_{1}<j_{2}<\cdots< j_{k}\leq n-2$.
By condition \mref{it:01002},
we consider the following two cases:

\textbf{Case $(1)$}: $d(1)=d'(1)=0$.

In this case, let $g: M_{n}\rightarrow M_{n}$ be a bijection such that $g(0)=0, g(1)=1$ and
  $g(b_{i_{\ell}})=b_{j_{\ell}}$ for $1\leq \ell\leq k$.
It is easy to see that $g$ is a lattice automorphism and $g(\Fix_{d}(M_{n}))=\Fix_{d'}(M_{n})$.
To prove that $gd=d'g$, let $x\in M_{n}$ and consider several cases.

If $x=1$, then $g(d(x))=g(d(1))=g(0)=0=d'(1)=d'(g(1))=d'(g(x))$, since $g(1)=1$, $g(0)=0$ and $d(1)=d'(1)=0$.

If  $x\in \Fix_{d}(M_{n})$, then $g(x)\in \Fix_{d'}(M_{n})$ and so
$d'(g(x))=g(x)=g(d(x))$.

If $x\in M_{n}\backslash (\Fix_{d}(M_{n})\cup \{1\})$, then $g(x)\in M_{n}\backslash (\Fix_{d'}(M_{n})\cup \{1\})$ and so  $d(x)=0$ and $d'(g(x))=0$
by Lemma \mref{lem:0100}. Consequently, we get $d'(g(x))=0=g(0)=g(d(x))$.

Therefore  $d'(g(x))=g(d(x))$ for any $x\in M_{n}$, which implies that $gd=d'g$. Hence $d\cong d'$.

\textbf{Case $(2)$}:   $d(1)\neq 0$ and $ d'(1)\neq 0$.

In this case, let $h: M_{n}\rightarrow M_{n}$ be a bijection such that $h(0)=0, h(1)=1$,
  $h(d(1))=d'(1)$  and $h(x)\in \{ b_{j_{1}}, b_{j_{2}}, \cdots, b_{j_{k}}\}\backslash \{d'(1)\}$
  for any
  $x \in\{ b_{i_{1}}, b_{i_{2}}, \cdots, b_{i_{k}}\}\backslash \{d(1)\}$.
It is easy to see that $h$ is  a lattice automorphism and $h(\Fix_{d}(M_{n}))=\Fix_{d'}(M_{n})$.
To prove that $hd=d'h$, let $x\in M_{n}$ and consider the following cases.

If $x=1$, then $h(d(x))=h(d(1))=d'(1)=d'(h(1))=d'(h(x))$, 
since $h(1)=1$.

If $x\in \Fix_{d}(M_{n})$, then
$h(x)\in \Fix_{d'}(M_{n})$ and so
$d'(h(x))=h(x)=h(d(x))$.

If  $x\in M_{n}\backslash (\Fix_{d}(M_{n})\cup \{1\})$, then
 $h(x)\in M_{n}\backslash (\Fix_{d'}(M_{n})\cup \{1\})$,
  and so  $d(x)=0$ and $d'(h(x))=0$ by Lemma \mref{lem:0100}.
Consequently, we obtain $d'(h(x))=0=h(0)=h(d(x))$.

Therefore  $d'(h(x))=h(d(x))$ for any $x\in M_{n}$, which implies that $hd=d'h$. Hence $d\cong d'$.
\end{proof}

Here is our classification of isomorphic classes of derivations on $M_n$.
\begin{theorem}
For any integer  $n\geq 3$, $|\DO(M_{n})|=2+\sum\limits_{k=1}^{n-2}(k+1)\binc{n-2}{k}$ and
there are exactly $2(n-1)$ isomorphic classes of derivations on $M_{n}$.
\mlabel{t:00}
\end{theorem}
\begin{proof}
Since $M_{3}$ is a $3$-element chain, the result is true for $n=3$ by Theorem \mref{the:001}. So we now assume $n\geq 4$.
Let $d\in \DO(M_{n})$. Notice that $0\in \Fix_{d}(M_{n})$.

If $|\Fix_{d}(M_{n})|=1$ which means that $\Fix_{d}(M_{n})=\{0\}$, then $d(x)=0$ for any $x\in M_{n}$, since $d(x)\in d(L)=\Fix_{d}(M_{n})$
by Corollary \mref{c:200}. Thus $d=\textbf{0}_{L}$ is the only choice in this case.

If $|\Fix_{d}(M_{n})|=n$ which means that $1\in \Fix_{d}(M_{n})$, then $d=\mrep_{L}$ by Proposition \mref{pro:201} \mref{it:2015} is the only choice in this case.

If $|\Fix_{d}(M_{n})|=k+1$ for some $1\leq k\leq n-2$, then $1\not\in \Fix_{d}(M_{n})$.
By Lemma \mref{lem:0100},
 $d(x)=0$ for any
$x\in  M_{n}\backslash (\Fix_{d}(M_{n})\cup \{1\})$ and
$d(1)$ is in $\Fix_{d}(M_{n})$.
Thus in this case, $d$ has exactly $(k+1)\binc{n-2}{k}$ choices.
Also, by Lemma \mref{lem:0101}, there are only two isomorphic classes of derivations in this case.

Summarizing the above cases, we conclude that $|\DO(M_{n})|=2+\sum\limits_{k=1}^{n-2}(k+1)\binc{n-2}{k}$
and there are exactly $2+2(n-2)=2(n-1)$ isomorphic classes of derivations on $M_{n}$.
\end{proof}

\section{The lattices of derivations}
\mlabel{sec:latder}
\nc{\opset}{\mathrm{O}}

In this section we study the set of derivations on a given lattice as a whole and consider lattice structures on the set. Such structures are obtained when conditions are imposed on either the lattice or on the derivations. Motivated by these evidences, we propose the conjectures that the poest of derivations on a lattice is again a lattice, and lattices are determined by their posets or lattices of derivations. 

\subsection{Lattice structures for derivations on distributive lattices}

Let $(L, \vee, \wedge,  0, 1)$ be a lattice and let $\opset(L)$ denote the set of all operators on $L$.
We define a relation $\preceq$ on $\opset(L)$.  For any $d, d'\in \opset(L)$,
define $d\preceq d'$ if $ d(x)\leq d'(x)$ for any $x\in L$.
It is easy to verify that  $\preceq$ is a partial order on $\opset(L)$ and
$\textbf{0}_{L}\preceq d \preceq \textbf{1}_{L}$ for any $d\in \opset(L)$, where 
$\textbf{1}_{L}$ is defined by
$\textbf{1}_{L}(x):=1$ for any $x\in L$. For any $d\in \DO(L)$, we have
$\textbf{0}_{L}\preceq d \preceq \mrep_{L}$ 
 since $0\leq d(x)\leq x$ for any $x\in L$.
 
We also define the following binary operations on $\opset(L)$. For $d, d'\in \opset(L)$, set
$$(d\vee d')(x):=d(x)\vee d'(x), \quad
 (d\cup d')(x):=x\wedge (d(1)\vee d'(1)),$$
 $$  (d\wedge d')(x):=d(x)\wedge d'(x), \quad (d\circ d')(x):=d(d'(x))
\quad \text{ for any } x\in L.$$
Of course the operation $\circ$ is just the composition. We retain the notion $\circ$ here to emphasize that it is a binary operation of operators. 

\begin{lemma} \mlabel{lem:701}
	Let $(L, \vee, \wedge,  0, 1)$ be a  lattice. Then $(\opset(L), \preceq, \textbf{0}_{L},\textbf{1}_{L} )$ is also a bounded lattice for which $d\vee d'$ and $d\wedge d'$ are, respectively, the least upper bound and the greatest lower bound of $d$ and $d'$.
\end{lemma}
\begin{proof}
	Since the class of all lattices is a variety and $\opset(L)$ is the direct product of $|L|$ copies of $L$, the lemma follows immedialtely from the usual notions of universal algebra \cite[Definition 7.8]{BS}.
\end{proof}

We next explore when these operators are derivations.

\begin{remark} Let $d, d'\in \DO(L)$.
	\begin{enumerate}
		\item By definition, the operator $d\cup d'=d_{d(1)\vee d'(1)}$ is an inner derivation.
		\item The operator $d\vee d'$ is not necessarily a derivation as shown by the following example. See Lemma~\mref{lem:40} for the case when $L$ is distributive.
		
Let $M_{5}=\{0, b_{1}, b_{2},  b_{3}, 1\}$ be the modular lattice in Remark \mref{re:000} and let $d=d_{b_{1}}, d'=d_{b_{3}}$, that is,
		$d(x)=x\wedge b_{1}$ and $d'(x)=x\wedge b_{3}$ for any $x\in M_{5}$.
		Then $d, d'\in \IDO(M_{5})$ by Proposition \mref{the:000}. Since $(d\vee d')(1)=d(1)\vee d'(1)=b_{1}\vee b_{3}=1$
		and $(d\vee d')(b_{2})=d(b_{2})\vee d'(b_{2})=0\vee 0=0$, we have by Proposition \mref{pro:201}
		that $d\vee d'\not\in \DO(M_{5})$.
		
		\item The operators $d\circ d'$ and $d\wedge d'$ are not necessarily derivations even if $(L, \vee, \wedge,  0, 1)$ is a Boolean lattice.
		For example, let $B_{8}=\{0, a, b, c, u, v, w, 1\}$ be the $8$-elements Boolean lattice in Example \mref{exm:000}.
		Then $\lambda^{(u)}, \lambda^{(v)}\in \DO(B_{8})$, but it is routine to verify that
		$\lambda^{(u)}\circ \lambda^{(v)}=\lambda^{(u)}\wedge \lambda^{(v)}=\lambda^{(a)}\not\in \DO(B_{8})$.
	\end{enumerate}
	\mlabel{remark:41}
\end{remark}

So in general, the set $\DO(L)$ is not closed under the operations $\vee, \circ$ or $\wedge$. We next consider the case when $L$ is a distributive lattice.

\begin{lemma}
	Let $(L, \vee, \wedge,  0, 1)$ be a distributive lattice. Then $d\vee d'$ is in $\DO(L)$ for any $d, d'\in \DO(L)$.
		\mlabel{it:lem401}
\mlabel{lem:40a}
\end{lemma}

\begin{proof} Assume that $(L, \vee, \wedge,  0, 1)$ is a distributive lattice.
	For $d, d'\in \DO(L)$ and $x, y\in L$,
	we have
	\begin{eqnarray*}
		(d\vee d')(x\wedge y)&=& d(x\wedge y)\vee d'(x\wedge y)\\
		&=& ((d(x)\wedge y)\vee (x\wedge d(y)))\vee ((d'(x)\wedge y)\vee (x\wedge d'(y)))\\
		&=&((d(x)\vee d'(x))\wedge y)\vee (x\wedge(d(y)\vee d'(y)))\\
		&=& ((d\vee d')(x)\wedge y)\vee (x\wedge(d\vee d')(y)).
	\end{eqnarray*}
	Thus  $d\vee d'\in\DO(L)$.
\end{proof}

\begin{theorem}
	Let $(L, \vee, \wedge,  0, 1)$ be a lattice. 
	\begin{enumerate}
		\item 
		If $d\vee d'$ and $d\wedge d'$ are in $\DO(L)$ for all $d$ and $d'$ in $\DO(L)$, then
		$(\DO(L), \vee, \wedge,\textbf{0}_L,\mrep_L)$ is a lattice. 
		\mlabel{it:4801}
		\item 	If $L$ is  finite and $d\vee d'$ is in $\DO(L)$  for all $d, d'\in \DO(L)$ $($or $d\wedge d'$ is in $\DO(L)$  for all $d, d'\in \DO(L)$$)$, then $(\DO(L),\preceq, \textbf{0}_L,\mrep_L)$ is a lattice.
		\mlabel{it:4802}
		\item If $L$ is a finite distributive lattice,  then $(\DO(L),\preceq, \textbf{0}_L,\mrep_L)$ is a lattice.
		\mlabel{it:4803}
	\end{enumerate}
	\mlabel{pro:480}
\end{theorem}
\begin{proof}
	\mnoindent
	\mref{it:4801}
	Assum that $d\vee d'$ and $d\wedge d'$ are in $\DO(L)$ for all $d, d'\in \DO(L)$. 
	Then $(\DO(L), \preceq)$ is a sublattice of the lattice $(\opset(L), \preceq)$ by Lemma \mref{lem:701}. Thus \mref{it:4801} holds.
	
	\mnoindent
	\mref{it:4802}
	Assume that $L$ is a finite lattice and $d\vee d'\in \DO(L)$ for all $d$ and $d'$ in $\DO(L)$.
	Since $\DO(L)$ is finite as a subset of the finite set $\opset(L)$, it follows that $\bigvee A:=\bigvee_{a\in A} a$ exists for every nonempty subset $A$ of $\DO(L)$. Noticing that $\bigvee \emptyset=\textbf{0}_L$, 
	 we have $(\DO(L),\preceq, \textbf{0}_L,\mrep_L)$ is a lattice by \cite[Theorem I 4.2]{BS}. The same argument applies if $d\wedge d'\in \DO(L)$ for all $d$ and $d'$ in $\DO(L)$.
	
	\mnoindent
	\mref{it:4803} follows immediately by Lemma \mref{lem:40a} and \mref{it:4802}.
\end{proof}

\begin{remark}
	Let $(L, \vee, \wedge,  0, 1)$ be a lattice. When $(\DO(L),\preceq, \textbf{0}_L,\mrep_L)$ is a lattice, it may not be a sublattice of
	$(\opset(L), \preceq)$. 
	
	For example, let $B_{8}=\{0, a, b, c, u, v, w, 1\}$ be the $8$-elements Boolean lattice in Example \mref{exm:000}.
	Then $(\DO(B_{8}),\preceq, \textbf{0}_{B_{8}},\mrep_{B_{8}})$ is a lattice by
	Theorem \mref{pro:480} \mref{it:4803}, since $B_{8}$ is a finite distributive lattice. But 
	$\lambda^{(u)}, \lambda^{(v)}\in \DO(B_{8})$ and
	$\lambda^{(u)}\wedge \lambda^{(v)}=\lambda^{(a)}\not\in \DO(B_{8})$ by Remark \mref{remark:41}. So $(\DO(B_{8}),\preceq, \textbf{0}_{B_{8}},\mrep_{B_{8}})$ is not a sublattice of $(\opset(B_{8}),\preceq)$.
\end{remark}

Recall that  a lattice is  \name{complete} if, for every subset $A$ of $L$, both $\bigvee A:=\bigvee_{a\in A} a$ and $\bigwedge A:=\bigwedge_{a\in A}a$ exist in $L$.
In a complete lattice $L$, there are two \name{infinite distributive laws} to consider, namely
\begin{equation}
	x\wedge \bigvee_{\alpha \in \Omega}y_{\alpha}=\bigvee_{\alpha \in \Omega}(x\wedge y_{\alpha}) ~~ and
	\mlabel{eq:301}
\end{equation}
\begin{equation}
	x\vee \bigwedge_{\alpha \in \Omega}y_{\alpha}=\bigwedge_{\alpha \in \Omega}(x\vee y_{\alpha})
	\mlabel{eq:302}
\end{equation}
for any $x,  y_{\alpha}\in L$ and any index set $\Omega$.
Unlike ordinary distributivity which is self-dual, these laws do not imply each other in general
\mcite{bly}.

Let $\{d_{i}\}_{i\in \Omega}$ be a family of operators on a complete lattice $L$. Define  operators  $\bigvee_{i\in \Omega}d_{i}$,
$\bigcup_{i\in \Omega}d_{i}$ and $\bigwedge_{i\in \Omega}d_{i}$ on $L$, respectively, by
$$\Big(\bigvee_{i\in \Omega}d_{i}\Big)(x):=\bigvee_{i\in \Omega}d_{i}(x), \
\Big(\bigcup_{i\in \Omega}d_{i}\Big)(x):=x\wedge(\bigvee_{i\in \Omega}d_{i}(1)), \
\Big(\bigwedge_{i\in \Omega}d_{i}\Big)(x):=\bigwedge_{i\in \Omega}d_{i}(x)$$
for any $x\in L$.

\begin{lemma}
	Let $(L, \vee, \wedge,  0, 1)$ be a complete lattice which satisfies the  infinite distributive law \meqref{eq:301}.
	\begin{enumerate}
		\item $\bigvee_{i\in \Omega}d_{i}\in \DO(L)$ for any family $\{d_{i}\}_{i\in \Omega}$ of  derivations on
		$L$.
		\mlabel{it:4111}
		\item  $\bigvee_{i\in \Omega}d_{i}=\bigcup_{i\in \Omega}d_{i}$ for any family $\{d_{i}\}_{i\in \Omega}$ of  isotone derivations on $L$.
		\mlabel{it:4112}
	\end{enumerate}
	\mlabel{lem:411}
\end{lemma}
\begin{proof}
	Let $(L, \vee, \wedge,  0, 1)$ be a complete lattice which satisfies the  infinite distributive law (\mref{eq:301}).
	
	\mnoindent
	\mref{it:4111} Assume that  $\{d_{i}\}_{i\in \Omega}$ is a family of  derivations on
	$L$. For any $x, y\in L$,
	we have
	\begin{eqnarray*}
		\Big(\bigvee_{i\in \Omega}d_{i}\Big)(x\wedge y)&=& \bigvee_{i\in \Omega}d_{i}(x\wedge y)\\
		&=&\bigvee_{i\in \Omega}  ((d_{i}(x)\wedge y)\vee (x\wedge d_{i}(y)))\\
		&=& \Big(\bigvee_{i\in \Omega}  ((d_{i}(x)\wedge y))\Big)\vee \bigvee_{i\in \Omega}(x\wedge d_{i}(y))\\
		&=&  \Big((\bigvee_{i\in \Omega}  d_{i}(x))\wedge y)\Big)\vee \Big(x\wedge \bigvee_{i\in \Omega}d_{i}(y)\Big)\\
		&=& \Big(\big(\bigvee_{i\in \Omega}  d_{i}\big)(x)\wedge y\Big)\vee \Big(x\wedge \big(\bigvee_{i\in \Omega}d_{i}\big)(y)\Big).
	\end{eqnarray*}
	Thus  $\bigvee_{i\in \Omega}d_{i}\in \DO(L)$.
	
	\mnoindent
	\mref{it:4112} Assume that  $\{d_{i}\}_{i\in \Omega}$ is a family of isotone derivations on
	$L$. For any $x\in L$, we have $d_{i}(x)=x\wedge d_{i}(1)$ by Proposition \mref{the:000}, and so
	$$\Big(\bigcup_{i\in \Omega}d_{i}\Big)(x)=x\wedge\Big(\bigvee_{i\in \Omega}d_{i}(1)\Big)=
	\bigvee_{i\in \Omega}(x\wedge d_{i}(1))= \bigvee_{i\in \Omega}d_{i}(x)=\Big(\bigvee_{i\in \Omega}d_{i}\Big)(x).$$ 
Thus
	$\bigvee_{i\in \Omega}d_{i}=\bigcup_{i\in \Omega}d_{i}$.
\end{proof}

\begin{theorem}
	Let $(L, \vee, \wedge,  0, 1)$ be a complete lattice which satisfies the  infinite distributive law in \meqref{eq:301}.
	Then $(\DO(L), \preceq, \textbf{0}_{L},\mrep_{L})$
	is a complete lattice.
	\mlabel{p:401}
\end{theorem}
\begin{proof}
	Let $\Omega$ be an index set and
	$\{d_{i}\}_{i\in \Omega}$ be a family of operators in $\DO(L)$.
	We shall show that  $\bigvee_{i\in \Omega}d_{i}$ is the least upper bound of $\{d_{i}\}_{i\in \Omega}$
	in the poset $(\DO(L), \preceq)$.
	
	Indeed, we have by Lemma \ref{lem:411} that
	$\bigvee_{i\in \Omega}d_{i}\in \DO(L)$. Also, for each $i\in \Omega$, we have
	$d_{i}(x)\leq \bigvee_{i\in \Omega}d_{i}(x)= (\bigvee_{i\in \Omega}d_{i})(x)$ for any $x\in L$ and so
	$d_{i}\preceq \bigvee_{i\in \Omega}d_{i}$. Thus  $\bigvee_{i\in \Omega}d_{i}$ is an upper bound of $\{d_{i}\}_{i\in \Omega}$.
	Finally, let
	$d'\in \DO(L)$ such that  $d_{i}\preceq d'$ for each $i\in \Omega$. Then  $d_{i}(x)\leq d'(x)$  for any $x\in L$, which implies that
	$(\bigvee_{i\in \Omega}d_{i})(x) =\bigvee_{i\in \Omega}d_{i}(x)\leq d'(x)$ and so $ \bigvee_{i\in \Omega}d_{i}\preceq d'$.
	Therefore we obtain that  $\bigvee_{i\in \Omega}d_{i}$ is
	the least upper bound of $\{d_{i}\}_{i\in \Omega}$
	in the poset $(\DO(L), \preceq)$.
	
	 Noting that $\bigvee \emptyset=\textbf{0}_{L}$, we get
	$(\DO(L), \preceq, \textbf{0}_{L},\mrep_{L})$
	is a complete lattice by \cite[Theorem I.4.2]{BS}.
\end{proof}

\subsection{Lattice structures on inner and other special derivations}

We next consider the lattice structure of inner derivations, leading to two realizations of any lattice as lattices of certain derivations. 
\begin{lemma} \mlabel{lem:301}
Let $(L, \vee, \wedge,  0, 1)$ be a  lattice. 
\begin{enumerate}
\item $d_{u}\cup d_{v}=d_{u\vee v}$ and $d_{u}\circ d_{v}=d_{u}\wedge d_{v}=d_{u\wedge v}$ for any $u, v\in L$.
\mlabel{it:1}
\item  \mlabel{it:2}
$d\cup d'$ and $d\wedge d'$ are in $\IDO(L)$ for any $d, d'\in \IDO(L)$.
\end{enumerate}
\end{lemma}
\begin{proof}
\mnoindent
\mref{it:1}
Let $u, v\in L$.
For any $x\in L$, since $d_{u}(1)=1\wedge u=u$ and $d_{v}(1)=1\wedge v=v$, we have
 $$(d_{u}\cup d_{v})(x)=x\wedge (d_{u}(1)\vee d_{v}(1))=x\wedge (u\vee v)=d_{u\vee v}(x),$$
  $$(d_{u}\circ d_{v})(x)=d_{u}( d_{v}(x))=d_{u} (x\wedge v)=(x\wedge v)\wedge u=x\wedge (u\wedge v)=d_{u\wedge v}(x),$$
 $$(d_{u}\wedge d_{v})(x)=d_{u}(x)\wedge d_{v}(x)=(x\wedge u) \wedge (x\wedge v)=x\wedge (u\wedge v)=d_{u\wedge v}(x).$$
 Thus  $d_{u}\cup d_{v}=d_{u\vee v}$ and $d_{u}\circ d_{v}=d_{u}\wedge d_{v}=d_{u\wedge v}$.

\mnoindent
\mref{it:2} follows immediately from~\mref{it:1} and Proposition \mref{the:000}.
\end{proof}
Now we give our first realization of a lattice as a lattice of derivations. 
\begin{proposition}
If $(L, \vee, \wedge,  0, 1)$ is a  lattice,
then $(\IDO(L), \cup, \wedge, \textbf{0}_{L},\mrep_{L})$
 is a lattice isomorphic to $L$.
\mlabel{pp:40}
\end{proposition}
\begin{proof}
Let $d$ and $d'$ be in $\IDO(L)$. Then $d\cup d'$ and $ d\wedge d'$ are in $ \IDO(L)$
by Lemma \mref{lem:301}.
It follows by Lemma \mref{lem:701} that 
$ d\wedge d'$ is the greatest lower bound of
$d$ and $d'$ in the poset  $(\IDO(L), \preceq)$.

We claim that
$d\cup d'$ is the least upper bound  of
$d$ and $d'$ in the poset  $(\IDO(L), \preceq)$.
In fact, first we have $d\preceq d\cup d'$ and
$d'\preceq d\cup d'$,
since 
$$d(x)=x\wedge d(1)\leq x\wedge (d(1)\vee d'(1))=(d\cup d')(x)$$ and 
$$d'(x)=x\wedge d'(1)\leq x\wedge (d(1)\vee d'(1))=(d\cup d')(x)$$ 
for any $x\in L$. Second, let $d''\in \IDO(L)$ such that $d\preceq d''$ and  $d'\preceq d''$. Then $d(1)\vee d'(1)\leq d''(1)$ and so for any $x\in L$,
$$(d\cup d')(x)=x\wedge (d(1)\vee d'(1))\leq x\wedge d''(1)=d''(x)$$
 by Proposition \mref{the:000}. Thus $d\cup d'\preceq d''$ and hence $d\cup d'$ is the least upper bound  of
 $d$ and $d'$ in the poset  $(\IDO(L), \preceq)$.
Therefore $(\IDO(L), \cup, \wedge, \textbf{0}_{L},\mrep_{L})$ is a lattice.

Define a map $f: \IDO(L)\rightarrow L$ by $f(d)=d(1)$
for any $d\in \IDO(L)$. By Corollary \mref{cor:200}, $f$ is a bijection. Also,   it is clear that $f(\textbf{0}_{L})=\textbf{0}_{L}(1)=0$ and
$f(\mrep_{L})=\mrep_{L}(1)=1$.
By  Lemma \mref{lem:301}, we have
$f(d_{u}\cup d_{v})=f(d_{u\vee v})=u\vee v=f(d_{u})\vee f(d_{v})$  and
$f(d_{u}\wedge d_{v})=f(d_{u\wedge v})=u\wedge v=f(d_{u})\wedge f(d_{v})$.
Thus $f$ is a lattice isomorphism.
\end{proof}

We next consider the case when $L$ is a distributive lattice.

\begin{lemma}	\mlabel{lem:40}
	Let $(L, \vee, \wedge,  0, 1)$ be a distributive lattice. Then $d\cup d'=d\vee d'$ for any $d, d'\in \IDO(L)$.
\end{lemma}
\begin{proof} Assume that $(L, \vee, \wedge,  0, 1)$ is a distributive lattice.
For $d, d'\in \IDO(L)$ and $x\in L$, we have $d(x)=x\wedge d(1)$ and $d'(x)=x\wedge d'(1)$
	by Proposition \mref{the:000} and so
	$$(d\vee d')(x)=d(x)\vee d'(x)=(x\wedge d(1))\vee (x\wedge d'(1))=x\wedge (d(1)\vee d'(1))
	=(d\cup d')(x).$$ 
Thus $d\cup d'=d\vee d'$.
\end{proof}

Then from Proposition~\mref{pp:40} and Lemma \mref{lem:40}, we obtain

\begin{corollary} $($\cite[Theorem 3.15]{xin2} \cite[Theorem~3.29]{xin1}$)$
If $(L, \vee, \wedge,  0, 1)$ is a distributive lattice,
then  $(\IDO(L), \vee, \wedge, \textbf{0}_{L},\mrep_{L})$ is a distributive lattice isomorphic to $L$.
\end{corollary}

Let $\chi^{(L)}=\{\chi^{(u)}~|~u\in L\}$, where $\chi^{(u)}$ is defined in Proposition \mref{pro:000}. We will show that
$(\chi^{(L)}, \preceq)$ is also a lattice isomorphic to $L$.

\begin{lemma}
Let $(L, \vee, \wedge,  0, 1)$ be a  lattice and $u, v\in L$. 
\begin{enumerate}
\item $\chi^{(u)}\vee\chi^{(v)}=\chi^{(u\vee v)}$  and
$\chi^{(u)}\wedge \chi^{(v)}=\chi^{(u\wedge v)}$.
\mlabel{it:3331}
\item  $\chi^{(u)}=\chi^{(v)}$ if and only if $u=v$.
\mlabel{it:3332}
\item   $\chi^{(u)}\circ \chi^{(v)}=\chi^{(v)}$ if $ v\neq 1$.
\mlabel{it:3333}
\end{enumerate}
 \mlabel{lem:333}
\end{lemma}
\begin{proof}
Let $L$ and $u, v\in L$ be as given.

\mnoindent
\mref{it:3331}
For any $x\in L$, we have
$$(\chi^{(u)}\vee\chi^{(v)})(x)=\chi^{(u)}(x)\vee\chi^{(v)}(x)=
    \begin{cases}
      u\vee v,  & \textrm{if}~ x=1; \\
      x,  & \textrm{otherwise}
    \end{cases}=\chi^{(u\vee v)}(x)$$
and
$$(\chi^{(u)}\wedge\chi^{(v)})(x)=\chi^{(u)}(x)\wedge\chi^{(v)}(x)=
    \begin{cases}
      u\wedge v,  & \textrm{if}~ x=1; \\
      x,  & \textrm{otherwise}
    \end{cases}=\chi^{(u\wedge v)}(x),$$
proving \mref{it:3331}.

\mnoindent
\mref{it:3332} It is clear that $u=v$ implies $\chi^{(u)}=\chi^{(v)}$. Conversely, if
$\chi^{(u)}=\chi^{(v)}$,  then $u=\chi^{(u)}(1)=\chi^{(v)}(1)=v$.

\mnoindent
\mref{it:3333} If $v\neq 1$, then for any $x\in L$, we have
$$(\chi^{(u)}\circ\chi^{(v)})(x)=\chi^{(u)}(\chi^{(v)}(x))=
    \begin{cases}
     \chi^{(u)}( v),  & \textrm{if}~ x=1; \\
     \chi^{(u)}( x),  & \textrm{otherwise}
    \end{cases}=
    \begin{cases}
     v,  & \textrm{if}~ x=1; \\
     x,  & \textrm{otherwise}
    \end{cases}=
    \chi^{( v)}(x),$$
proving \mref{it:3333}.
\end{proof}
Now we give our second realization of a lattice as a lattice of derivations. 
\begin{proposition}
If $(L, \vee, \wedge,  0, 1)$ is a  lattice,
then  $(\chi^{(L)}, \preceq)$ is a sublattice of
$(\opset(L), \preceq)$ that is isomorphic to $L$.
\mlabel{pp:440}
\end{proposition}
\begin{proof}
Assume that $(L, \vee, \wedge,  0, 1)$ is a  lattice and  $u, v\in L$. Then  
$\chi^{(u)}\vee\chi^{(v)}=\chi^{(u\vee v)}\in \chi^{(L)}$  and
$\chi^{(u)}\wedge \chi^{(v)}=\chi^{(u\wedge v)}\in \chi^{(L)}$ by Lemma \mref{lem:333}. Thus $(\chi^{(L)}, \preceq)$ is a sublattice of $(\opset(L), \preceq)$ by Lemma \mref{lem:701}.

Define a map $f: L\rightarrow \chi^{(L)}$ by $f(u)=\chi^{(u)}$
for any $u\in L$. By Lemma \mref{lem:333}, $f$ is an injective homomorphism. Also, it is clear that $f$ is surjective. Hence
 $f$ is a lattice isomorphism.
\end{proof}

\begin{proposition}
	Let $(L, \vee, \wedge,  0, 1)$ be a  lattice.
	If $(\DO(L), \vee, \wedge, \textbf{0}_{L},\mrep_{L})$
	is a distributive lattice,  then  $L$ is also distributive.
	\mlabel{p:4002}
\end{proposition}
\begin{proof} Assume that $(\DO(L), \vee, \wedge, \textbf{0}_{L},\mrep_{L})$
	is a distributive lattice. Then 
	$(\chi^{(L)}, \preceq)$ is a sublattice of
	$(\DO(L), \vee, \wedge, \textbf{0}_{L},\mrep_{L})$ and  
	$(\chi^{(L)}, \preceq)$
	is  isomorphic to $L$ by Proposition \mref{pp:440}.
	It follows that
	$(\chi^{(L)}, \preceq)$ is a distributive lattice and hence $L$ is distributive.
\end{proof}

A nonempty subset $F$ of a lattice $L$ is called a \name{filter} \mcite{bly} of $L$ if it satisfies:
$(i)$  $a, b\in F$ implies $a\wedge b\in F$ and
$(ii)$ $a\in F$, $c\in L$ and $a\leq c$ imply $c\in F$.

\begin{proposition}
	Let $(L, \vee, \wedge,  0, 1)$ be a  lattice.
	If $(\DO(L), \vee, \wedge, \textbf{0}_{L},\mrep_{L})$
	is a lattice,  then  $\chi^{(L)}$ is a filter of the lattice $\DO(L)$.
	\mlabel{p:4003}
\end{proposition}
\begin{proof} Assume that $(\DO(L), \vee, \wedge, \textbf{0}_{L},\mrep_{L})$
	is a  lattice. It is clear that
	$\chi^{(L)}$ is a nonempty subset of $\DO(L)$. Also, by Lemma \mref{lem:333}, $\chi^{(L)}$ is closed under meet $\wedge$.
	
	Finally, assume that $d\in \DO(L)$ such that $\chi^{(u)}\preceq d$ for some $u\in L$. Then $L\backslash\{1\}\subseteq \Fix_{d}(L)$. In fact, for any
	$x\in L\backslash \{1\}$, we have $x=\chi^{(u)}(x)\leq d(x)$ and so $d(x)=x$, since $d(x)\leq x$  by Proposition \mref{pro:201}.
	It follows that $x\in \Fix_{d}(L)$ and hence  $L\backslash\{1\}\subseteq \Fix_{d}(L)$. Consequently, we have
	$d\in \chi^{(L)}$. Therefore $\chi^{(L)}$ is a filter of the lattice $\DO(L)$.
\end{proof}

\subsection{Lattice structures for derivations on specific lattices}
We now show that derivations on some concrete lattices form lattices. 
From Theorem~\mref{pro:480}, we know that
$(\DO(L), \preceq, \textbf{0}_{L},\mrep_{L})$
is a  lattice if $L$ is a finite chain or $L=M_{4}$. Next, we will show that
$(\DO(L), \preceq, \textbf{0}_{L},\mrep_{L})$
is a sublattice of 	$(\opset(L), \preceq)$ if $L$ is a finite chain or $L=M_{4}$, and prove that
$(\DO(M_{n}), \preceq, \textbf{0}_{L},\mrep_{L})$
is also a  lattice when $n\geq 5$.

\begin{lemma}
	Let  $(L, \vee, \wedge,  0, 1)$ be a  lattice
	and $d, d'\in \DO(L)$. Then
	$\Fix_{d}(L)\cap \Fix_{d'}(L)= \Fix_{d\wedge d'}(L)$.
	\mlabel{lem:500}
\end{lemma}
\begin{proof}
	Assume that $d, d'\in \DO(L)$.
	If $x\in \Fix_{d}(L)\cap \Fix_{d'}(L)$, then
	$d(x)=d'(x)=x$ and so
	$(d\wedge d')(x)=d(x)\wedge d'(x)=x$, i.e, $x\in \Fix_{d\wedge d'}(L)$. Thus
		$\Fix_{d}(L)\cap \Fix_{d'}(L)\subseteq \Fix_{d\wedge d'}(L)$.
	
Conversely, if $x\in \Fix_{d\wedge d'}(L)$, then $x=(d\wedge d')(x)=d(x)\wedge d'(x)$ and so $d(x)=x=d'(x)$ since $d(x)\leq x$ and $d'(x)\leq x$. Thus $x\in \Fix_{d}(L)\cap \Fix_{d'}(L)$. Therefore we get
	$\Fix_{d}(L)\cap \Fix_{d'}(L)= \Fix_{d\wedge d'}(L)$.
\end{proof}

\begin{lemma}
	Let  $(L, \vee, \wedge,  0, 1)$ be a finite  chain and $d, d'\in \DO(L)$. Then
	$d\wedge d'$ is in $\DO(L)$.
	\mlabel{lem:501}
\end{lemma}
\begin{proof}
	Assume that $(L, \vee, \wedge,  0, 1)$ is a finite  chain and $d, d'\in \DO(L)$. 
Let
$u= \max_{x\in L}\{(d\wedge d')(x)\}$. Then by Lemma \mref{lem:500} and Corollary \mref{c:200} we have 
\begin{eqnarray*}
	u&=& \max_{x\in L}\{\Fix_{d\wedge d'}(L)\}\\
	&=& \max_{x\in L}\{\Fix_{d}(L)\cap \Fix_{d'}(L)\}\\
	&=& \max_{x\in L}\{\Fix_{d}(L)\}\wedge \max_{x\in L}\{\Fix_{d'}(L)\}\\
	&=& \max_{x\in L}\{d(L)\}\wedge \max_{x\in L}\{d'(L)\}.
\end{eqnarray*} 
Without loss of generality, assume that $u=\max_{x\in L}\{d(x)\}\leq \max_{x\in L}\{d'(x)\}$, and let
$u'=\max_{x\in L}\{d'(x)\}$.

For any $x\leq u$, 
since $u= \max_{x\in L}\{d(x)\}$ and  $u\leq u'$,   we have $d(x)=d'(x)=x$ by Lemma \mref{lem:0001}, and so $(d\wedge d')(x)=d(x)\wedge d'(x)=x$.

For any $v, w\in L$ with $u< w\leq v$, we have  $d(v)\leq d(w)$ by Lemma  \mref{lem:0001}.

 If 
$w\leq u'$, then $d'(w)=w$ by Lemma  \mref{lem:0001}. It follows that $(d\wedge d')(v)=d(v)\wedge d'(v)\leq d(v)\leq d(w)=d(w)\wedge d'(w)=(d\wedge d')(w)$, since $d(w)\leq w$.
 
 If $u'<w$, then
$u'< w\leq v$, and so $d'(v)\leq d'(w)$  by Lemma  \mref{lem:0001}. It follows that $(d\wedge d')(v)=d(v)\wedge d'(v)\leq d(w)\wedge d'(w)=(d\wedge d')(w)$.
 
 Summarizing the above arguments, by Lemma \mref{lem:0001} we obtain $d\wedge d'\in \DO(L)$.
\end{proof}

\begin{lemma}
Let $n\geq 4$ and $d, d'\in \DO(M_{n})$. Then
	 $d\wedge d'$ is in $\DO(M_{n})$.
\mlabel{l:4000}
\end{lemma}
\begin{proof}
Assume that $n\geq 4$ and $d, d'\in \DO(M_{n})$.
Then $\Fix_{d}(M_{n})\cap \Fix_{d'}(M_{n})= \Fix_{d\wedge d'}(M_{n})$ by Lemma \mref{lem:500}. To prove that $d\wedge d'\in \DO(M_{n})$, consider the following two cases.

First assume  $1\in \Fix_{d}(M_{n})$ or $1\in \Fix_{d'}(M_{n})$. Then
		$d=\mrep_{M_{n}}$ or $d'=\mrep_{M_{n}}$ by Proposition \mref{pro:201} \mref{it:2014} and so $d\wedge d'=d'\in \DO(M_{n})$ or $d\wedge d'=d\in \DO(M_{n})$. 
		
Next assume $1\not\in \Fix_{d}(M_{n})$ and  $1\not\in \Fix_{d'}(M_{n})$.
		Then $d(1), d'(1)\in M_{n}\backslash \{1\}$. 
		If $d(1)=d'(1)$, then 
		$$(d\wedge d')(1)=d(1)\wedge d'(1)=d(1)=d'(1)\in \Fix_{d}(M_{n})\cap \Fix_{d'}(M_{n})= \Fix_{d\wedge d'}(M_{n}).$$
		If $d(1)\neq d'(1)$, then
$$(d\wedge d')(1)=d(1)\wedge d'(1)=0\in \Fix_{d}(M_{n})\cap \Fix_{d'}(M_{n})= \Fix_{d\wedge d'}(M_{n}).$$ 
	Also, for each $w\in M_n\backslash (\Fix_{d\wedge d'}(M_n)\cup\{1\})$, we have $w\in M_n\backslash (\Fix_{d}(M_n)\cup\{1\})$ or $w\in M_n\backslash (\Fix_{ d'}(M_n)\cup\{1\})$ and so	
		$d(w)=0$ or $d'(w)=0$ by Lemma \mref{lem:0100}. Thus $(d\wedge d')(w)=d(w)\wedge d'(w)=0$. Consequently, by Lemma \mref{lem:0100} we obtain   $d\wedge d'\in \DO(M_{n})$.
\end{proof}

\begin{proposition}
 $(\DO(M_{n}), \preceq, \textbf{0}_{M_{n}},\mrep_{M_{n}})$
	is  a  lattice	for any integer $n\geq 3$.
	\mlabel{p:4000}
\end{proposition}
\begin{proof}
Since $M_{3}$ is a $3$-element chain, $(\DO(M_{3}), \preceq, \textbf{0}_{M_{3}},\mrep_{M_{3}})$
is  a  lattice by Theorem \mref{pro:480}.

Let $n\geq 4$ and $d, d'\in \DO(M_{n})$.
Then  $d\wedge d'\in \DO(M_{n})$ by Lemma \mref{l:4000}. It follows by Theorem \mref{pro:480} \mref{it:4802} that $(\DO(M_{n}), \preceq, \textbf{0}_{M_{n}},\mrep_{M_{n}})$
is  a  lattice.
\end{proof}

\begin{proposition}
	Let  $(L, \vee, \wedge,  0, 1)$ be a finite  chain or $L=M_{4}$. Then $(\DO(L), \preceq)$
	is a sublattice of 	$(\opset(L), \preceq)$. 		
	\mlabel{pro:502}
\end{proposition}
\begin{proof}
Assume that  $L$ is  a finite  chain or $L=M_{4}$. Then $L$ is a finite distributive lattie. It follows by Lemma \mref{lem:40a}, Theorem \mref{pro:480}, Lemma \mref{lem:501} and Lemma \mref{l:4000}	that $(\DO(L), \preceq)$
is a sublattice of 	$(\opset(L), \preceq)$. 		
\end{proof}

The following examples demonstrate a rich lattice structure on the set of lattice derivations.

For  a finite set $X=\{x_{1}, x_{2}, \cdots, x_{n}\}$ and a map $\varphi: X\rightarrow X$, we will write $\varphi$ as
$$ \left( \begin{matrix}
	x_{1}          & x_{2}            & \cdots & x_{n} \\
	\varphi(x_{1}) &   \varphi(x_{2}) & \cdots &   \varphi(x_{n})
\end{matrix}
\right)$$

\begin{example}
\begin{enumerate}
\item Let $C_{3}=\{0, u,  1\}$ be the $3$-element chain with $0< u< 1$. Then by Theorem \mref{the:001}, $|\DO(C_{3})|=4$ and
$\DO(C_{3})=\{\mathbf{0}_{C_{3}}, \mrep_{C_{3}}, \varphi_{1}, \varphi_{2}\}$, where
$$\varphi_{1}= \left( \begin{matrix}
        0     & u        & 1\\
        0     & u        & 0
          \end{matrix}
          \right), \quad
 \varphi_{2}= \left( \begin{matrix}
        0     & u        & 1\\
        0     & u     & u
          \end{matrix}
          \right).  $$
By Proposition \mref{pro:502}, $(\DO(C_{3}), \preceq, \textbf{0}_{C_{3}},\mrep_{C_{3}})$ is a sublattice of 	$(\opset(C_{3}), \preceq)$.  Moreover, it is easy to see that  $\DO(C_{3})$ is a $4$-element chain with $\mathbf{0}_{C_{3}}< \varphi_{1}<\varphi_{2}<\mrep_{C_{3}}$,
$\IDO(C_{3})=\{\mathbf{0}_{C_{3}}, \mrep_{C_{3}},  \varphi_{2}\}$, and $\chi^{(C_{3})}=\{\varphi_{1}, \varphi_{2}, \mrep_{C_{3}}\}$.
\mlabel{it:401}
\item   Let $C_{4}=\{0, u, v, 1\}$ be the $4$-element chain with $0< u<v< 1$. Then by Theorem \mref{the:001}, $|\DO(C_{4})|=8$. Moreover, we have
$\IDO(C_{4})=\{\mathbf{0}_{C_{4}}, \mrep_{C_{4}},  x_{3},  x_{6}\}$, $\chi^{(C_{4})}=\{x_{4}, x_{5}, x_{6}, \mrep_{C_{4}}\}$ and
$\DO(C_{4})=\{\mathbf{0}_{C_{4}}, \mrep_{C_{4}}, x_{1}, x_{2}, x_{3}, x_{4}, x_{5}, x_{6}\}$, where
$$x_{1}= \left( \begin{matrix}
        0     & u     & v   & 1\\
        0     & u     & 0   & 0
          \end{matrix}
          \right), \quad
 x_{2}= \left( \begin{matrix}
        0     & u     & v   & 1\\
        0     & u     & u   & 0
          \end{matrix}
          \right), \quad
  x_{3}= \left( \begin{matrix}
        0     & u     & v   & 1\\
        0     & u     & u   & u
          \end{matrix}
          \right),      $$
 $$x_{4}= \left( \begin{matrix}
        0     & u     & v   & 1\\
        0     & u     & v   & 0
          \end{matrix}
          \right), \quad
 x_{5}= \left( \begin{matrix}
        0     & u     & v   & 1\\
        0     & u     & v   & u
          \end{matrix}
          \right), \quad and \quad
  x_{6}= \left( \begin{matrix}
        0     & u     & v   & 1\\
        0     & u     & v   & v
          \end{matrix}
          \right).      $$
 By Proposition \mref{pro:502}, $(\DO(C_{4}), \preceq, \textbf{0}_{C_{4}},\mrep_{C_{4}})$ is a sublattice of 	$(\opset(C_{4}), \preceq)$.         
 It is easy to see that the Hasse diagram of $\DO(C_{4})$ is
 $$
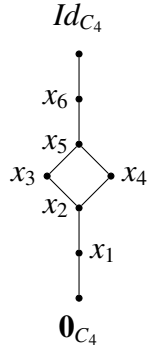
\begin{tikzpicture}
\tikzstyle{every node}=[draw,circle,fill=black,node distance=0.6cm,
minimum size=0.8pt, inner sep=0.8pt]
\node[circle] (1)                        [label=above :   $\mrep_{C_{4}}$]{};
\node[circle] (2)   [below   of=1]                      [label=left :   $x_{6}$]{};
\node[circle] (3) [below   of=2]     [label=left : $x_{5}$]{};
\node[circle] (4) [below right of=3]             [label=right  : $x_{4}$] {};
\node[circle] (5)   [below left  of=3]                      [label=left :   $x_{3}$]{};
\node[circle] (6) [below right  of=5]     [label=left : $x_{2}$]{};
\node[circle] (7) [below  of=6]             [label=right  : $x_{1}$] {};
\node[circle] (8) [below   of=7]     [label=below: $\mathbf{0}_{C_{4}}$] {};

\draw[-] (1) --   (2); \draw[-] (2) --   (3); \draw[-] (3) --   (4);
 \draw[-] (3) --   (5); \draw[-] (4) --   (6);
\draw[-] (5) --   (6); \draw[-] (6) --   (7);
\draw[-] (7) --   (8);
\end{tikzpicture}
$$
\mlabel{it:402}
\vspace{-.5cm}
\item  By Theorem \mref{t:00}, we have $|\DO(M_{4})|=9$. Moreover, we have
$\IDO(M_{4})=\{\mathbf{0}_{M_{4}}, \mrep_{M_{4}},  y_{2},  y_{4}\}$, $\chi^{(M_{4})}=\{y_{5}, y_{6}, y_{7}, \mrep_{M_{4}}\}$ and
$\DO(M_{4})=\{\mathbf{0}_{M_{4}}, \mrep_{M_{4}}, y_{1}, y_{2}, y_{3}, y_{4}, y_{5}, y_{6}, y_{7}\}$, where
$$\quad y_{1}= \left( \begin{matrix}
        0     & b_{1}     & b_{2}   & 1\\
        0     & b_{1}     & 0       & 0
          \end{matrix}
          \right),
 y_{2}= \left( \begin{matrix}
        0     & b_{1}     & b_{2}   & 1\\
        0     & b_{1}     & 0       & b_{1}
          \end{matrix}
          \right),
  y_{3}= \left( \begin{matrix}
        0     & b_{1}     & b_{2}   & 1\\
        0     & 0         & b_{2}   & 0
          \end{matrix}
          \right),     y_{4}= \left( \begin{matrix}
        0     & b_{1}     & b_{2}   & 1\\
        0     & 0         & b_{2}   & b_{2}
          \end{matrix}
          \right), $$
 $$
 y_{5}= \left( \begin{matrix}
        0     & b_{1}     & b_{2}   & 1\\
        0     & b_{1}     & b_{2}   & 0
          \end{matrix}
          \right), \quad
 y_{6}= \left( \begin{matrix}
        0     & b_{1}     & b_{2}   & 1\\
        0     & b_{1}     & b_{2}   & b_{1}
          \end{matrix}
          \right), \quad
              and \quad
  y_{7}= \left( \begin{matrix}
        0     & b_{1}     & b_{2}   & 1\\
        0     & b_{1}     & b_{2}   & b_{2}
          \end{matrix}
          \right).
              $$
By Proposition \mref{pro:502}, $(\DO(M_{4}), \preceq, \textbf{0}_{M_{4}},\mrep_{M_{4}})$ is a sublattice of 	$(\opset(M_{4}), \preceq)$.               
 It is easy to see that the Hasse diagram of $\DO(M_{4})$ is given by
 $$
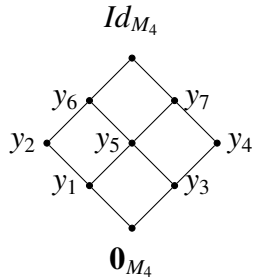
\begin{tikzpicture}
\tikzstyle{every node}=[draw,circle,fill=black,node distance=0.8cm,
minimum size=0.8pt, inner sep=0.8pt]
\node[circle] (1)                        [label=above :   $\mrep_{M_{4}}$]{};
\node[circle] (2)   [below left  of=1]                      [label=left :   $y_{6}$]{};
\node[circle] (3) [below right  of=1]                       [label=right : $y_{7}$]{};
\node[circle] (4) [below left of=2]                         [label=left  : $y_{2}$] {};
\node[circle] (5)   [below right  of=2]                      [label=left :   $y_{5}$]{};
\node[circle] (6) [below right  of=3]                      [label=right : $y_{4}$]{};
\node[circle] (7) [below left of=5]             [label=left  : $y_{1}$] {};
\node[circle] (8) [below right of=5]             [label=right  : $y_{3}$] {};
\node[circle] (9) [below right  of=7]     [label=below: $\mathbf{0}_{M_{4}}$] {};

\draw[-] (1) --   (2); \draw[-] (1) --   (3); \draw[-] (2) --   (4);
 \draw[-] (2) --   (5); \draw[-] (3) --   (5);\draw[-] (3) --   (6);
\draw[-] (4) --   (7); \draw[-] (5) --   (7);\draw[-] (5) --   (8); \draw[-] (6) --   (8);
\draw[-] (7) --   (9);\draw[-] (8) --   (9);
\end{tikzpicture}
$$
\mlabel{it:403}
\end{enumerate}
 \mlabel{exa:400}
\end{example}
\vspace{-1cm}

Proposition \mref{p:4001}  shows that the lattice of derivations is rarely a chain.
\begin{proposition}
	Let $(L, \vee, \wedge,  0, 1)$ be a  lattice.
	Then $(\DO(L), \preceq, \textbf{0}_{L},\mrep_{L})$
	is a chain if and only if $L$ is a chain with $|L|\leq 3$.
	\mlabel{p:4001}
\end{proposition}
\begin{proof}
	If $L$ is a chain with $|L|\leq 3$, then $(\DO(L), \preceq, \textbf{0}_{L},\mrep_{L})$
	is a chain by Proposition \mref{p:2000} and Example \mref{exa:400} \mref{it:401}.
	
	Conversely, assume that $(\DO(L), \preceq, \textbf{0}_{L},\mrep_{L})$
	is a chain. If $L$ is not a chain, then there exist $u, v\in L$ such that $u$ and $v$ are incomparable.
	By Proposition \mref{pro:000}, we know that $\chi^{(u)}, \chi^{(v)}\in \DO(L)$.
	Since $\chi^{(u)}(1)=u$ and $ \chi^{(v)}(1)=v$, the derivations $\chi^{(u)}$ and $ \chi^{(v)}$
	are incomparable, a contradiction. Thus we get that $L$ is a chain.
	
	If $L$ is a chain with $|L|\geq 4$, then there exist $w, z\in L\backslash \{0, 1\}$ such that $w< z$.
	We have $d_{w}\in \DO(L) $ (here $d_{w}$ is the inner derivation) and $\lambda^{(z)}\in \DO(L)$ by Example \mref{exm:000}.
	But $d_{w}$ and $\lambda^{(z)}$ are incomparable, since   $d_{w}(z)=w\wedge z=w, \lambda^{(z)}(z)=z$, $d_{w}(1)=w\wedge 1=w$ and $ \lambda^{(z)}(1)=0$. Therefore,  $L$ is a chain with $|L|\leq 3$.
\end{proof}

Based on the results in this section, it seems reasonable to pose the following conjectures.

\begin{conjecture}
\mlabel{qu:do}
\begin{enumerate}
	\item \mlabel{it:do1}
For any lattice $(L,\vee, \wedge, 0,1)$, the poset $(\DO(L), \preceq, \textbf{0}_{L},\mrep_{L})$ is a lattice.
\item \mlabel{it:do2}
For any two lattices $L$ and $L'$, if $(\DO(L), \preceq, \textbf{0}_{L},\mrep_{L})$ and
$(\DO(L'), \preceq, \textbf{0}_{L'},\mrep_{L'})$ are isomorphic lattices, then $L$ and $L'$ are isomorphic lattices.
\end{enumerate}
Conjecture \ref{it:do2} depends on Conjecture \ref{it:do1}. As a standalone conjecture, we give 
\begin{enumerate}\addtocounter{enumi}{2}
	\item 
	For any two lattices $L$ and $L'$, if $(\DO(L), \preceq, \textbf{0}_{L},\mrep_{L})$ and
	$(\DO(L'), \preceq, \textbf{0}_{L'},\mrep_{L'})$ are isomorphic posets, then $L$ and $L'$ are isomorphic lattices.
\end{enumerate}
\mlabel{cj:00}	
\end{conjecture}

\noindent
{\bf Acknowledgments.}
This work is supported by National Natural Science Foundation of China (Grant Nos. 11771190, 11801239).

\end{document}